\documentclass[12pt,righttagstyle]{amsart}

\usepackage{amssymb,latexsym,amsmath,epsfig}

\def\pf{\begin{proof}}

\def\C{{\mathbb C}}

\def\O{{\mathcal O}}

\newtheorem{thm}{Theorem}[section]
\newtheorem{prop}[thm]{Proposition}
\newtheorem{cor}[thm]{Corollary}
\newtheorem{lem}[thm]{Lemma}
\newtheorem{defn}[thm]{Definition}
\newtheorem{remark}[thm]{Remark}

\def\O{\Omega}

\def\la{\lambda}

\def \a{ \alpha}
\def\da{\delta}

\def\ga{\gamma}
\def\Ga{\Gamma}

\def\ba{\beta}

\newcommand{\bprop} {\begin{proposition}}
\newcommand{\eprop} {\end{proposition}}
\newcommand{\btheo} {\begin{theorem}}
\newcommand{\etheo} {\end{theorem}}
\newcommand{\blem} {\begin{lemma}}
\newcommand{\elem} {\end{lemma}}
\newcommand{\bcor} {\begin{corollary}}
\newcommand{\ecor} {\end{corollary}}

\newcommand{\Be}{\begin{equation}}
\newcommand{\Ee}{\end{equation}}
\newcommand{\Bea}{\begin{eqnarray}}
\newcommand{\Eea}{\end{eqnarray}}
\newcommand{\Bes}{\begin{equation*}}
\newcommand{\Ees}{\end{equation*}}
\newcommand{\Beas}{\begin{eqnarray*}}
\newcommand{\Eeas}{\end{eqnarray*}}
\newcommand{\Ba}{\begin{array}}
\newcommand{\Ea}{\end{array}}
\def\R{\mathbb{R}}
\def\C{\mathbb{C}}

\def\F{\mathcal F}

\scrollmode

\begin{document}

\title[Lebesgue mixed norm estimates for Bergman Projectors]{Lebesgue mixed norm estimates for Bergman projectors:
from tube domains over homogeneous cones to Homogeneous Siegel Domains of Type II }

\author[D. B\'ekoll\'e]{David B\'ekoll\'e}
\address{Department of Mathematics, Faculty of Science, University of Ngaound\'er\'e\\ P.O. Box 454, Ngaound\'er\'e, Cameroon }
\email{{\tt dbekolle@univ-ndere.cm}}
\author[J. Gonessa]{Jocelyn Gonessa}
\address{Universit\'e de Bangui, Facult\'e des Sciences, D\'epartement de math\'ematiques et Informatique, BP. 908, Bangui, R\'epublique Centrafricaine}
\email{gonessa.jocelyn@gmail.com}
\author[C. Nana]{Cyrille Nana}
\address{Faculty of Science, Department of Mathematics, University of Buea, P.O. Box 63, Buea, Cameroon}
\email{{\tt nana.cyrille@ubuea.cm}}

\begin{abstract}
We present a transference principle of Lebesgue mixed norm estimates for Bergman projectors from tube domains over homogeneous cones to homogeneous Siegel domains of type II associated to the same cones.  This principle implies improvements of these estimates for homogeneous Siegel domains of type II associated with Lorentz cones, e.g. the Pyateckii-Shapiro Siegel domain of type II.
\end{abstract}
\maketitle
\section{Introduction }
Let  $D$ be a domain in $\mathbb C^n$ and $dv$ the Lebesgue measure defined in $\mathbb C^n.$
We denote by $P$ the Bergman projector i.e., the orthogonal projector of the Hilbert space $L^2 (D, dv)$ onto its closed subspace $A^2 (D, dv)$
consisting of holomorphic functions on $D.$ It is well-known that $P$ is an integral operator defined on $L^p (D, dv)$ whose kernel $B(.,.),$
called the Bergman kernel, is the reproducing kernel of  $A^2 (D, dv).$ In this work, we consider the case where $D$ is a homogeneous Siegel
domain of type II and we are interested in the values of $p\geq 1$ for which the Bergman projector $P$ can be extended as a bounded
operator on $L^p (D, dv).$ More generally, we investigate the values $1\leq p, q \leq \infty$ for which the Bergman projector extends
to a bounded operator on Lebesgue mixed norm spaces $L^{p, q} (D).$

In fact, C. Nana \cite{Nana} determined a range of values $1\leq p, q \leq \infty$ for which the Bergman projector of a homogeneous
Siegel domain of type II extends as a bounded operator on Lebesgue mixed norm spaces $L^{p, q} (D).$ He even considered the case where
the Lebesgue measure $dv$ is replaced by standard weighted measures. Earlier in a joint work \cite{NT} with B. Trojan, the same author
considered the particular case of tube domains over homogeneous cones (homogeneous Siegel domains of type I). The purpose of the present
paper is to present a transference principle to deduce mixed norm estimates for Bergman projectors on homogeneous Siegel domains of
type II from analogous estimates on tube domains over associated cones. As an application, the results of \cite{Nana} can be
obtained as consequences of the results of \cite{NT}.

\section{Description of homogeneous cones and homogeneous Siegel domains of type II. Statement of the main results}
In this section, we recall the description of a homogeneous cone within the framework of $T$-algebras.
Next, we introduce homogeneous Siegel domains of type II and state our main results.

\subsection{Homogeneous cones}
We use the same notations as in \cite{Chua} and \cite{NT}. We denote by $\mathcal U$ a (real)
matrix algebra of rank $r$ with canonical decomposition
\begin{equation*}
\mathcal U =  \bigoplus\limits_{1\leq i, j \leq r}  \mathcal U_{ij}
\end{equation*}
such that $\mathcal U_{ij} \mathcal U_{jk} \subset \mathcal U_{ik}$ and $\mathcal U_{ij} \mathcal U_{lk} = \{0\} $ if $j\neq l.$ We assume that $\mathcal U$ has a structure of $T$-algebra (in the sense of \cite{V}) in which an involution is given by $x\mapsto x^\star.$ This structure implies that the subspaces $\mathcal U_{ij}$ satisfy: $\mathcal U_{ii} = \mathbb Rc_i$ where $c_i^2 = c_i$ and $dim \hskip 1truemm \mathcal U_{ij} = n_{ij} = n_{ji}.$ Also, the matrix
\begin{equation*}
\mathbf e = \sum_{j=1}^r c_j
\end{equation*}
is a unit element for the algebra $\mathcal U.$

Let $\rho$ be the unique isomorphism from $\mathcal U_{ii}$ onto $\mathbb R$ with $\rho (c_i) = 1$ for all $i=1,...,r.$ We shall consider the subalgebra
\begin{equation*}
\mathcal T =  \bigoplus\limits_{1\leq i \leq j \leq r}  \mathcal U_{i j}
\end{equation*}
of $\mathcal U$ consisting of upper triangular matrices and let
\begin{equation*}
H = \{t\in \mathcal T: \rho (t_{ii}) >  0, \hskip 2truemm i=1,...,r\}
\end{equation*}
be the subgroup of upper triangular matrices whose diagonal elements are positive.

Denote by $V$ the vector space of "Hermitian matrices" in $\mathcal U$
\begin{equation*}
V=  \{x\in \mathcal U: \hskip 2truemm x^\star = x\}.
\end{equation*}
If we set
\begin{equation*}
n_i = \sum_{j=1}^{i-1} n_{ji}, \quad \quad m_i = \sum_{j=i+1}^{r} n_{ij},
\end{equation*}
then
\begin{equation}\label{dim}
dim \hskip 1truemm V = n= r+\sum_{i=1}^r m_i = r+\sum_{i=1}^r n_i.
\end{equation}
The vector space $V$ becomes a Euclidean space with the inner product
\begin{equation*}
(x|y) = tr \hskip 1truemm (xy^\star)
\end{equation*}
where
\begin{equation*}
tr \hskip 1truemm (x) = \sum_{i=1}^r \rho (x_{ii}).
\end{equation*}
Next we define
\begin{equation*}
\Omega=  \{ss^\star: \hskip 2truemm s\in H\}.
\end{equation*}
By a theorem of Vinberg (\cite[p. 384]{V}), $\Omega$ is an open convex homogeneous cone containing no entire straight lines, in which the group $H$ acts simply transitively via the transformations
\begin{equation}\label{simplytransitive}
\pi (w): uu^\star \mapsto \pi (w)[uu^\star] = (wu)(u^\star w^\star) \quad \quad (w, u \in H).
\end{equation}
Thus, to every element $y\in \Omega$ corresponds a unique $t\in H$ such that
$$y=\pi (t)[\textbf e].$$
Like in \cite{NT}, we shall adopt the notation:
$$t\cdot \textbf e = \pi (t)[\textbf e].$$
We shall assume that $\Omega$ is irreducible, and hence rank $(\Omega)=r.$ All homogeneous convex cones can be constructed in this way
(\cite[p. 397]{V}).

As in \cite{NT}, we denote by $Q_j$ the fundamental rational functions in $\Omega$ given by
\begin{equation*}
Q_j (y) = \rho (t_{jj})^2, \quad \quad {\rm when} \hskip 2truemm y=t \cdot {\mathbf e} \in \Omega.
\end{equation*}
We consider the matrix algebra  $\mathcal U'$ which differs from $\mathcal U$ only on its grading, in the sense that
\begin{equation*}
\mathcal U'_{ij} = \mathcal U_{r+1-i, r+1-j} \quad \quad (i, j = 1,...,r).
\end{equation*}
It is proved in \cite{V} that $\mathcal U'$ is also a $T$-algebra and $V' = V$ where $V'$ is the subspace of $\mathcal U'$
consisting of Hermitian matrices. We define accordingly its subalgebra
\begin{equation*}
\mathcal T' =  \bigoplus\limits_{1\leq i \leq j \leq r}  \mathcal U'_{ij}
\end{equation*}
of $\mathcal U$ consisting of lower triangular matrices
and the subgroup $H'$ of $\mathcal T'$ whose diagonal elements are positive. We have
$$\mathcal T' = \{t^\star: t\in \mathcal T\} \quad {\rm and} \quad H' = H^* = \{t^\star: t\in H\}.$$
The corresponding homogeneous cone coincides with the dual cone of $\Omega,$ namely
\begin{equation*}
\Omega^*=  \{\xi \in V': (x|\xi) > 0, \hskip 2truemm \forall x \in \overline {\Omega} \setminus \{0\}\}.
\end{equation*}
One also has
\begin{equation*}
\Omega^*=  \{t^\star t: \hskip 2truemm t\in H\}.
\end{equation*}
(See \cite[p. 390]{V}).

For $\xi = t^\star t \in \Omega^*,$ we shall define
\begin{equation*}
Q^*_j (\xi) = \rho (t_{jj}^2).
\end{equation*}
The group $H'$ acts simply transitively on the cone $\Omega^*$ via the transformations
\begin{eqnarray}\label{starsimplytransitive}
\pi (w^\star): u^\star u\mapsto \pi (w^\star)[u^\star u] = (w^\star u^\star)(uw) \,\,\,(w^\star, u^\star \in H').
\end{eqnarray}
We write
$$t^\star \cdot \textbf e = \pi (t^\star)[\textbf e]\quad \quad (t^\star \in H').$$
We have  the following identity.
\begin{equation}\label{star}
Q_j^* (t^\star \cdot \mathbf e) = Q_j (t\cdot \mathbf e).
\end{equation}
In the sequel, we shall use the following notations: for all $x\in \Omega, \hskip 2truemm \xi \in \Omega^*$ and $\alpha = (\alpha_1, \alpha_2,...,\alpha_r) \in \mathbb R^r,$
\begin{equation*}
Q^\alpha (x) = \prod_{j=1}^r Q_j^{\alpha_j} (x) \quad  {\rm and} \quad (Q^*)^\alpha (\xi) = \prod_{j=1}^r (Q^*_j)^{\alpha_j} (\xi).
\end{equation*}
We identify a real number $\beta$ with the vector $(\beta, ...,\beta)\in \mathbb R^r$ and we write
\begin{equation*}
Q^\beta (x) = \prod_{j=1}^r Q_j^{\beta} (x) \quad  {\rm and} \quad (Q^*)^\beta (\xi) = \prod_{j=1}^r (Q^*_j)^{\beta} (\xi),
\end{equation*}
\begin{equation*}
Q^{\alpha + \beta} (x) = \prod_{j=1}^r Q_j^{\alpha_j + \beta} (x) \quad  {\rm and} \quad (Q^*)^{\alpha +\beta} (\xi) = \prod_{j=1}^r (Q^*_j)^{\alpha_j + \beta} (\xi).
\end{equation*}

 We put $\tau = (\tau_1, \tau_2,...,\tau_r) \in \mathbb R^r$ with
$$\tau_i = 1+\frac 12 (m_i +n_i).$$
Let $x\in \Omega,$ we have for $j=1,...,r$
\begin{equation} {\label{Qj}}
Q_j (\pi (t)[x]) = Q_j (t\cdot \mathbf e)Q_j (x).
\end{equation}
Therefore, for any $t\in H,$
$$Q^\tau (\pi (t)[x]) = det \hskip 1truemm \pi (t)Q^\tau (x)$$
since
$$det \hskip 1truemm \pi (t) = Q^\tau (t\cdot {\mathbf e}).$$
(See \cite[p. 388]{V}). The above properties are also valid if we replace $Q_j$ by $Q^*_j$ and $x\in \Omega$ by $\xi \in \Omega^*.$
In particular, for all $\xi\in \O^*$ and $t^\star \in H',$ we have for $j=1,...,r$

\begin{equation} {\label{QjStar}}
Q^*_j (\pi (t^\star)[\xi]) = Q^*_j (t^\star\cdot \mathbf e)Q^*_j (\xi).
\end{equation}

\subsection{Homogeneous Siegel domains of type II}
Let $V^{\mathbb C} = V+iV$ be the complexification of $V.$ Then each element of $V^{\mathbb C}$ is identified with a vector in $\mathbb C^n.$
The coordinates of a point $z\in \mathbb C^n$ are arranged in the form
\begin{equation}\label{z}
z=(z_{11}, z_2, z_{22},...,z_r, z_{rr})
\end{equation}
where
\begin{equation}\label{z_j}
z_j = (z_{1j},...,z_{j-1, j}), \quad \quad j=2,...,r
\end{equation}
and
\begin{equation}\label{z_{ij}}
z_{jj} \in \mathbb C, \quad z_{ij} = (z^{(1)}_{ij},...,z^{(n_{ij})}_{ij})\in \mathbb C^{n_{ij}}, \quad 1\leq i < j \leq r.
\end{equation}
For all $j=1,...,r$ we denote $e_{jj} = z,$ where $z_{jj} = 1$ and the other coordinates are equal to zero and we denote
\begin{equation*}
e=\sum_{j=1}^r e_{jj} = (1, 0, 1,...,0, 1).
\end{equation*}
Let $m\in \mathbb N.$ For each row vector $u\in \mathbb C^m,$ we denote $u'$ the transpose of $u.$ Given $m\times m$ Hermitian matrices
$\widetilde{H}_{11},\widetilde{H}_2,\widetilde{H}_{22},...,\widetilde{H}_r, \widetilde{H}_{rr}$ such that for every $j=1,...,r,$ we have
$$u\widetilde{H}_{jj} \bar v' \in \mathbb C, \quad \quad u\widetilde{H}_j \bar v' \in \mathbb C^{n_j},$$
we define a $\Omega$-Hermitian, homogeneous form $F: \mathbb C^m \times \mathbb C^m \rightarrow \mathbb C^n$ as
\begin{equation}\label{form}
F (u, v) = (u\widetilde{H}_{11} \bar v',u\widetilde{H}_2 \bar v',u\widetilde{H}_{22} \bar v'...,u\widetilde{H}_r \bar v', u\widetilde{H}_{rr} \bar v'), \quad \quad (u, v) \in \mathbb C^m \times \mathbb C^m
\end{equation}
such that
\begin{enumerate}
\item[(i)]
$F(u, u) \in \overline {\Omega}:$
\item[(ii)]
$F(u, u) = 0$ if and only if $u=0;$
\item[(iii)]
for every $t\in H,$ there exists $\tilde t \in GL (m, \mathbb C)$ such that
\begin{equation}\label{tilde}
t\cdot F(u,u) = F(\tilde tu, \tilde tu).
\end{equation}
\end{enumerate}
The point set
\begin{equation}\label{Siegel}
D(\Omega, F) = \{(z, u) \in \mathbb C^n \times \mathbb C^m: \Im m \hskip 1truemm z - F(u, u) \in \Omega\}
\end{equation}
in $\mathbb C^{n+m}$ is called a Siegel domain of type II associated to the open convex homogeneous cone $\Omega$ and to the $\Omega-$Hermitian,
homogeneous form $F.$ Recall that if $m=0,$ the domain $D$ is a tube type Siegel domain or a homogeneous Siegel domain of type I,
associated with the cone $\Omega,$ or the tube domain over the homogeneous cone $\Omega,$ considered by the authors of \cite{NT}.

Using (\ref{z}), we write
$$F(u, u) = (F_{11}(u, u), F_2(u, u), F_{22}(u, u),...,F_r(u, u), F_{rr}(u, u))$$
where for $i=1,...,r$ and $j=2,...,r,$
$$F_{ii} (u, u) = u\widetilde{H}_{ii} \bar u', \quad F_j (u, u) = u\widetilde {H}_j \hat u'= (F_{1j} (u, u),...,F_{j-1, j} (u, u))$$
and for $1\leq i < j \leq r$ and $\lambda=1,...,n_{ij},$
$$F_{ij} (u, u)= (F^{(1)}_{ij} (u, u),...,F^{(n_{ij})}_{ij} (u, u)), \quad F^{(\lambda)}_{ij} (u, u)=u\widetilde{H}^{(\lambda)}_{ij} \bar u'.$$
The space $\mathbb C^m$ decomposes into the direct sum of subspaces $\mathbb C^{b_1} \oplus...\oplus \mathbb C^{b_r}$ on which are
concentrated the Hermitian forms $F_{jj},$ that is, with appropriate coordinates, we have for $i=1,...,r,$
\begin{equation}\label{decomp}
\widetilde{H}_{ii} = \mbox{diag} (0_{(b_1)},...,0_{(b_{i-1}}, I_{(b_i)},0_{(b_{i+1})},...,0_{(b_r)})
\end{equation}
where $0_{(b_k)}$ and $I_{(b_k)}$ denote respectively the null matrix and the identity matrix of the vector space
$\mathbb C^{b_k}$ for all $k=1,...,r.$ (See for instance \cite[ pp. 127-129]{X}.)

In the sequel, we denote $b$ the vector
$$b=(b_1,...,b_r)\in \mathbb N^r.$$
and we denote $dv$ the Lebesgue measure in $\mathbb C^m.$  Let $\nu = (\nu_1,...,\nu_r) \in \mathbb R^r.$ For all $(x+iy, u) \in D,$ we shall consider the measure
$$dV_\nu (x+iy, u) = Q^{\nu-\frac b2 - \tau} (y - F(u,u))dxdydv(u).$$

 We denote by $L^p_\nu (D), \hskip 2truemm 1\leq p \leq \infty,$ the Lebesgue space $L^p (D, dV_\nu (z, u)).$ The weighted Bergman
 space $A^p_\nu (D)$ is the (closed) subspace of $L^p_\nu (D)$ consisting of holomorphic functions. In order to have a non-trivial subspace,
 we take $\nu = (\nu_1,...,\nu_r) \in \mathbb R^r$ such that $\nu_i > \frac {m_i  + b_i}2, \hskip 2truemm i=1,...,r.$ (See \cite{Nana}.)

The orthogonal projector of the Hilbert space $L^2_\nu (D)$ on its closed subspace $A^2_\nu (D)$ is the weighted Bergman
projector $P_\nu.$ We recall that $P_\nu$ is defined by the integral
$$P_\nu f(z, u) = \int_D B_\nu ((z, u), (w, v))f(w, v)dV_\nu (w, v), \quad (z, u) \in D,$$
where for a suitable constant $d_{\nu, b},$
$$B_\nu ((z, u), (w, v))=d_{\nu, b}Q^{-\nu - \frac b2 -\tau} \left(\frac {z-\bar w}{2i}-F(u,v)\right)$$
is the weighted Bergman kernel i.e., the reproducing kernel of $A^2_\nu (D).$ (See \cite[Proposition II.5]{BT}.) The scalar product $\langle \cdot, \cdot \rangle_\nu$ is given by
$$\langle f, g\rangle_\nu = \int_{D} f(z, u)\overline{g(z, u)}dV_\nu (z, u).$$

Let us now introduce mixed norm spaces. For $1\leq p \leq \infty$ and $1\leq q < \infty,$ let $L^{p, q}_\nu (D)$ be the space of measurable functions on $D$ such that
$$||f||_{L^{p, q}_\nu (D)} := \left(\int_{\mathbb C^m} \int_{\Omega + F(u, u) } \left(\int_V |f(x+iy, u)|^p dx\right)^{\frac qp} Q^{\nu-\tau-\frac b2} (y-F(u, u)dydv(u)\right)^{\frac 1q}$$
is finite (with obvious modification if $p=\infty)$. As before, we call $A^{p, q}_\nu (D)$ the (closed) subspace of $L^{p, q}_\nu (D)$ consisting of holomorphic functions.
Note that for $p=q,$ the Lebesgue mixed norm space $L^{p, q}_\nu (D)$ coincides with the Lebesgue space $L^p_\nu (D)$ and the mixed norm Bergman space $A^{p, q}_\nu (D)$ coincides with the Bergman space $A^p_\nu (D).$

The unweighted case corresponds to $\nu = \tau + \frac b2.$

\subsection{Example}\label{type 2} The homogeneous (non-symmetric) Siegel domain of type II introduced by Pyateckii-Shapiro is associated to the spherical cone
$$\Gamma:= \left\{(y_{11}, y_{12}, y_{22})\in \mathbb R^3: Q_1 (y) = y_{11}> 0, Q_2 (y) = y_{22}-\frac {(y_{12})^2}{y_{11}} > 0\right\}$$
and to the $\Gamma-$Hermitian, homogeneous form
\begin{eqnarray*}
\begin{array}{clcr}
F:&\mathbb C^2  \longrightarrow \mathbb C^3&\\
&(u,v)\mapsto F(u, v) = (0, 0, u\bar v).&
\end{array}
\end{eqnarray*}
In this domain $D(\Gamma, F),$ we have
$$n=3, \hskip 2truemm r=2, \hskip 2truemm n_1=m_2=0, \hskip 2truemm n_2=m_1=1,\,\,\,b_1=0,\,\,\,b_2=1,  \hskip 2truemm \tau =\left(\frac 32, \frac 32\right).$$
The (unweighted) Bergman kernel of $D(\Gamma, F)$ has the following expression:
$$B ((z, u), (w, v)) = CQ_1^{-3}\left(\frac {z-\overline {w}}{2i}-F(u, v)\right)Q_2^{-4}\left(\frac {z-\overline {w}}{2i}-F(u, v)\right)$$
which can be written
$$B ((z, u), (w, v)) = C\left(\frac {z_{11}-\overline {w_{11}}}{2i}\right)^{-3}\left(\frac {z_{22}-\overline {w_{22}}}{2i}-u\bar v-\frac{\left(\frac {z_{12}-\overline {w_{12}}}{2i}\right)^2}{\frac {z_{11}-\overline {w_{11}}}{2i}}\right)^{-4}.$$
For $\nu =(\nu_1, \nu_2) \in \mathbb R^2$ such that $\nu_j > \frac 12, \hskip 2truemm j=1, 2,$ the associated (weighted) Bergman kernel is given by
$$B_\nu ((z, u), (w, v)) = d_\nu Q_1^{-\nu_1-\frac 32}\left(\frac {z-\overline {w}}{2i}-F(u, v)\right)Q_2^{-\nu_2-2}\left(\frac {z-\overline {w}}{2i}-F(u, v)\right).$$

\subsection{Statement of the results}
The main result of our paper is the following.

\begin{thm}
Let $\nu = (\nu_1,...,\nu_r) \in \mathbb R^r$ such that $\nu_j > \frac {m_j +n_j + b_j}2, \hskip 2truemm j=1,...,r.$ Assume that the Bergman projector $\mathbb P_{\nu - \frac b2}$ of the tube domain $T_\Omega$ over the homogeneous cone $\Omega$ is bounded on $L^{p, q}_{\nu - \frac b2} (T_\Omega).$ Then the Bergman projector $P_\nu$ of the homogeneous Siegel domain $D$ of type II associated to $\Omega$ and to the $\Omega-$Hermitian homogeneous form $F$ is bounded on $L^{p, q}_\nu (D).$
\end{thm}

As a consequence, the following result of \cite{Nana} for $D$ is a
consequence of the corresponding result of \cite{NT} for tube domains over homogeneous cones (see Theorem \ref{boundedness} below).

\begin{thm}
Let $\nu = (\nu_1,...,\nu_r) \in \mathbb R^r$ such that $\nu_j > \frac {m_j +n_j + b_j}2, \hskip 2truemm j=1,...,r.$ We set
$q_\nu := 1+\min \limits_{1\leq j \leq r} \frac {\nu_j - \frac {m_j}2 -\frac {b_j}2}{\frac {n_j}2}.$ The Bergman projector
$P_\nu$ extends to a bounded operator on $L^{p, q}_\nu (D)$ for
\begin{eqnarray*}
\left \{
\begin{array}{clcr}
0&\leq \frac 1p &\leq \frac 12\\
\frac 1{q_\nu p'}&<\frac 1q &< 1-\frac 1{q_\nu p'}
\end{array}
\right.
\quad \quad or \quad \quad
\left \{
\begin{array}{clcr}
\frac 12&\leq \frac 1p &\leq 1\\
\frac 1{q_\nu p}&<\frac 1q &< 1-\frac 1{q_\nu p}
\end{array}
\right.
.
\end{eqnarray*}
\end{thm}

Our theorem  implies improvements of $L^{p, q}_\nu$ estimates for Bergman projectors in homogeneous
Siegel domains of type II associated to Lorentz cones for some particular values of $\nu \in \mathbb R^2$. An interesting case is the Pyateckii-Shapiro Siegel domain of type II
defined in Example \ref{type 2}.
For this domain, the problem under study  was investigated in \cite[section 6]{Go} and a particular case of Theorem 2.1 was used there. We point out that the underlying spherical cone is isomorphic to the Lorentz cone of $\mathbb R^3.$
The $L^{p, q}_\nu$ estimates (with $\nu$ real, $\nu =(\nu,\cdots, \nu))$ for the Bergman projectors on tube
domains over Lorentz cones are now completely settled after the works of \cite{BBPR}, \cite{BBGR} and \cite{BBGRS}
(cf. also \cite{BBGNPR} and \cite{B}), and the recent proof of the $l^2$-decoupling conjecture by Bourgain and Demeter \cite{BD}. This goal is achieved via a particular case of the following more general theorem where $\nu= (\nu_1, \nu_2)$ is a vector of $\mathbb R^2.$ We denote by $\Lambda_n$ the Lorentz cone of $\mathbb R^n, \hskip 2truemm n\geq 3,$ and we adopt the following notations.
$$p_\nu =1+\frac {\nu_2 +\frac n2}{(\frac n2 -1 -\nu_2)_+};$$
$$q_\nu= 1+\frac {\nu_2}{\frac {n}2-1};$$
$$q_\nu (p)= p_\sharp q_\nu \quad {\rm with} \quad p_\sharp = \min (p, p');$$
$$\tilde q_{\nu, p}
=\frac {\nu_2 + \frac n2-1}{(\frac n{2p'} -1)_+}.$$

\begin{thm}
Let $\nu= (\nu_1, \nu_2)\in \mathbb R^2$ such that $\nu_1 > \frac n2 -1$ and $\nu_2 >0.$ The weighted Bergman $P_\nu$ of the tube domain $T_{\Lambda_n}$ is bounded in $L^{p, q}_\nu (T_{\Lambda_n})$ for the following values of $p, q$ and $\nu.$
\begin{enumerate}
\item
$\frac {n-2}{\nu_2 +\frac n2 -1}<p<\frac {n-2}{\frac n2 -1-\nu_2 }, \hskip 2truemm q'_\nu (p)<q<q_\nu (p)$ provided $0<\nu_2 < \frac n2 -1;$
\item
$1\leq p \leq \infty, \hskip 2truemm q'_\nu (p)<q<q_\nu (p)$ provided $\nu_2 \geq \frac n2 -1;$
\item
$2\leq p\leq \frac {2n}{n-2}, \hskip 2truemm p<p_\nu$ and $q'_\nu (p) <q<2q_\nu;$
\item
$p_\nu >p>\frac {2n}{n-2}$ and $2<q<\widetilde q_{\nu, p}$
 provided $\frac {n-2}{2n}<\nu_2 < \frac n2 -1;$
\item
$p>\frac {2n}{n-2}$ and $2<q<\widetilde q_{\nu, p}$ provided $\nu_2 \geq \frac n2 -1;$
\item
the couples $(p, q)$ obtained by complex interpolation from the previous couples.
\end{enumerate}
\end{thm}

\vskip 20truemm

\begin{figure}[tbph]
	\centering
	\includegraphics[width=0.7\linewidth]{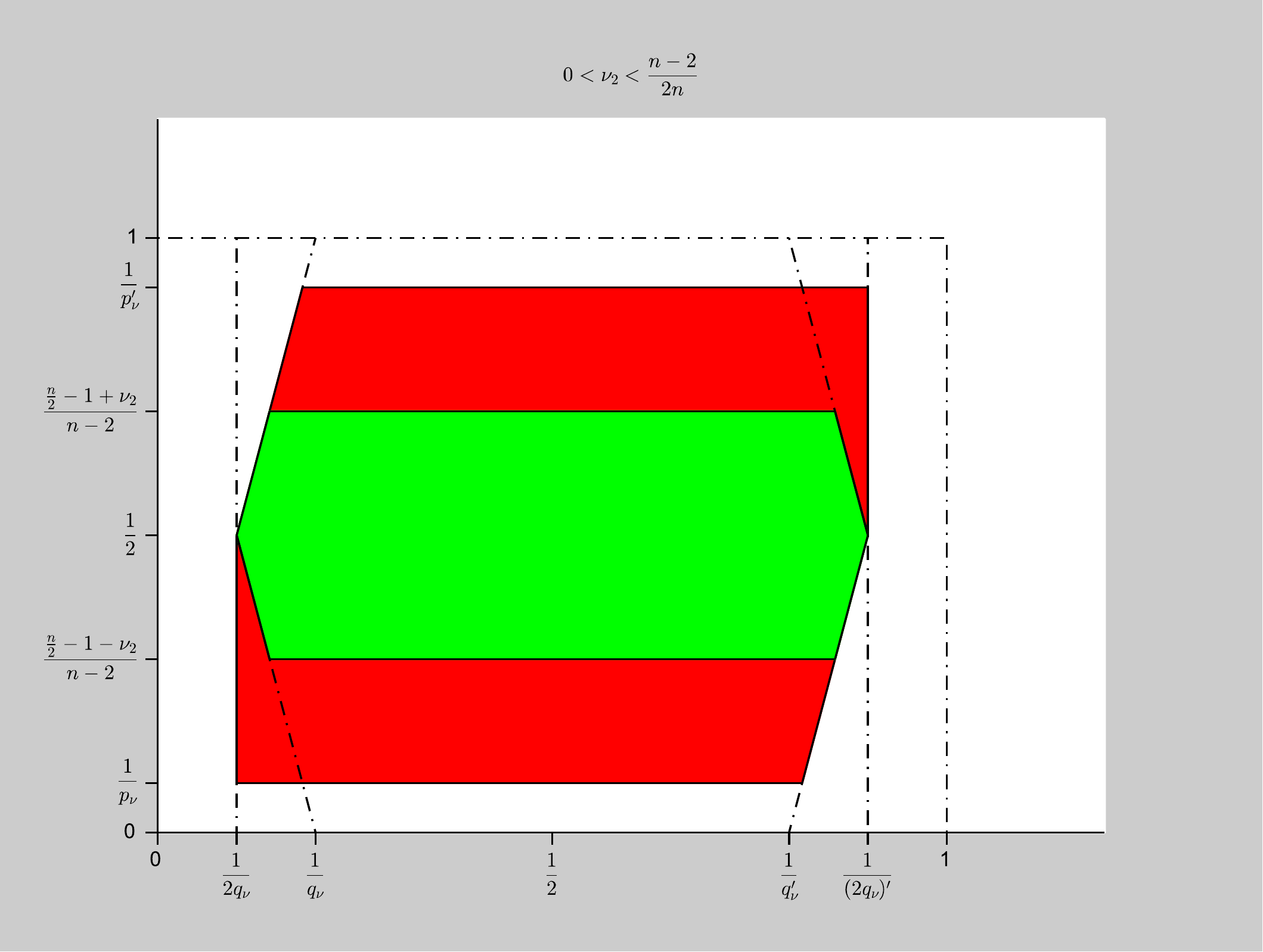}
	\caption{}
	\label{fig1}
\end{figure}

\vskip 20truemm

\begin{figure}[tbph]
	\centering
	\includegraphics[width=0.7\linewidth]{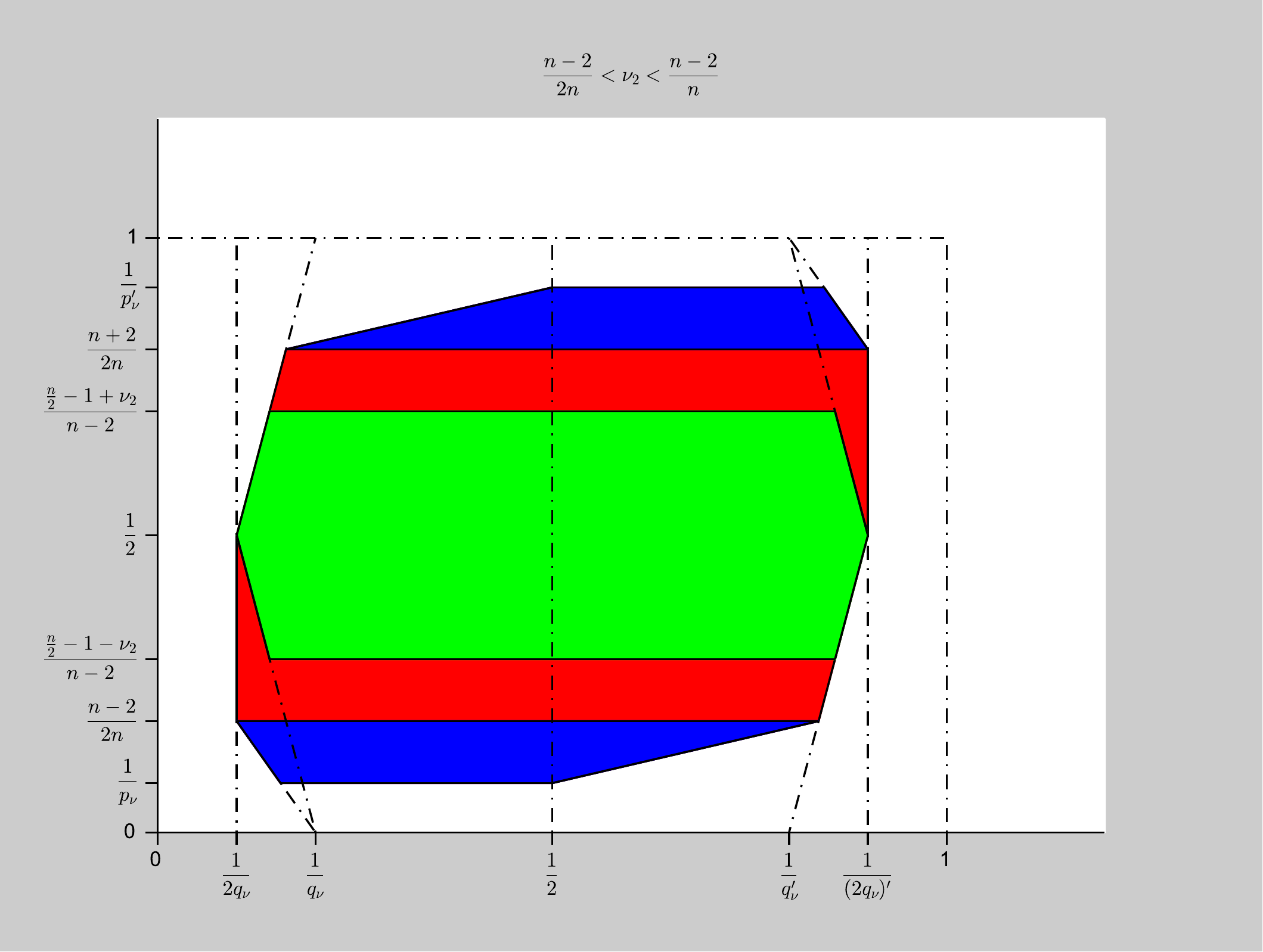}
	\caption{}
	\label{fig2}
\end{figure}

\vskip 20truemm

\begin{figure}[tbph]
	\centering
	\includegraphics[width=0.7\linewidth]{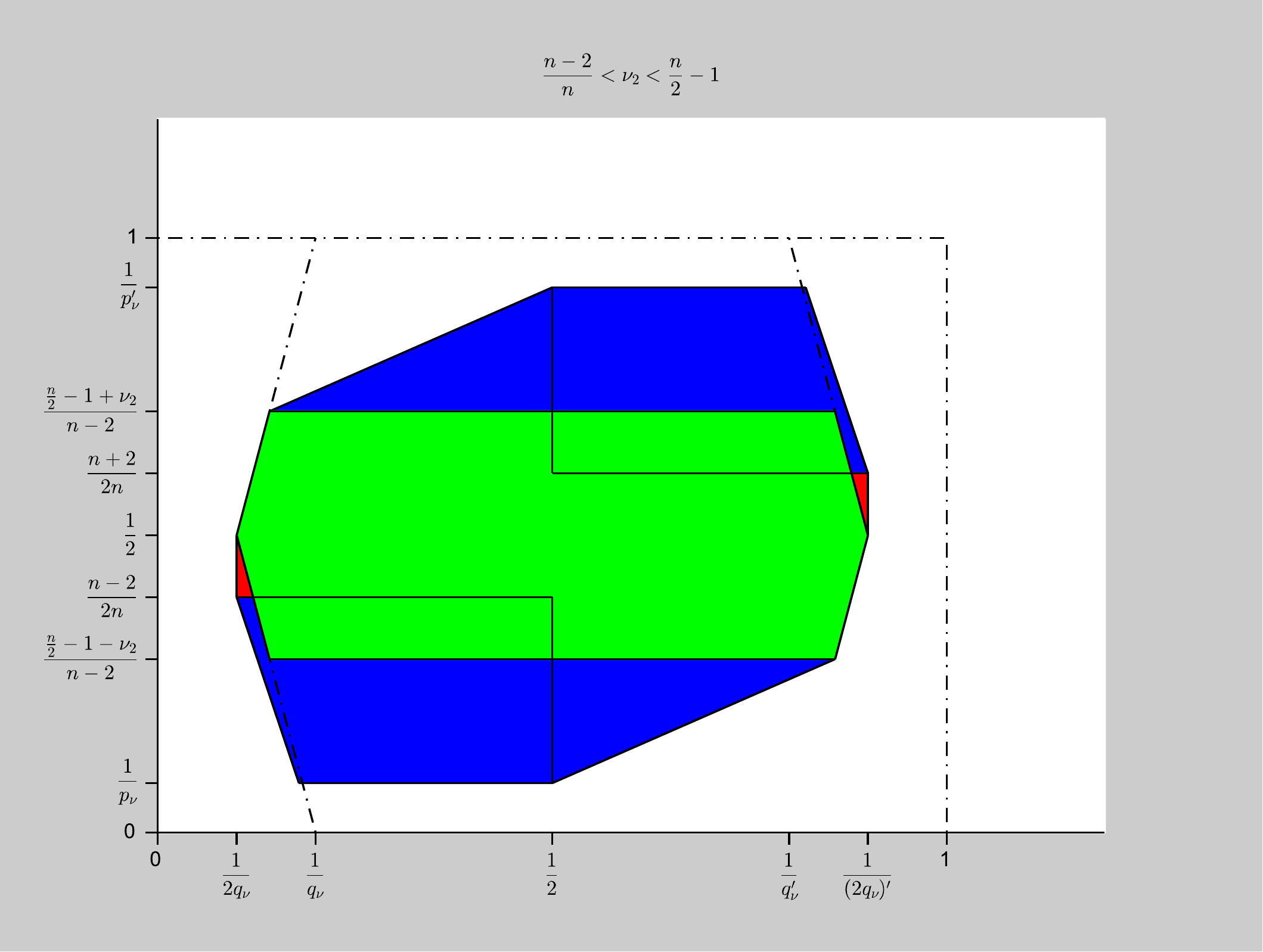}
	\caption{}
	\label{fig3}
\end{figure}

\vskip 2truemm

\begin{figure}[tbph]
	\centering
	\includegraphics[width=0.7\linewidth]{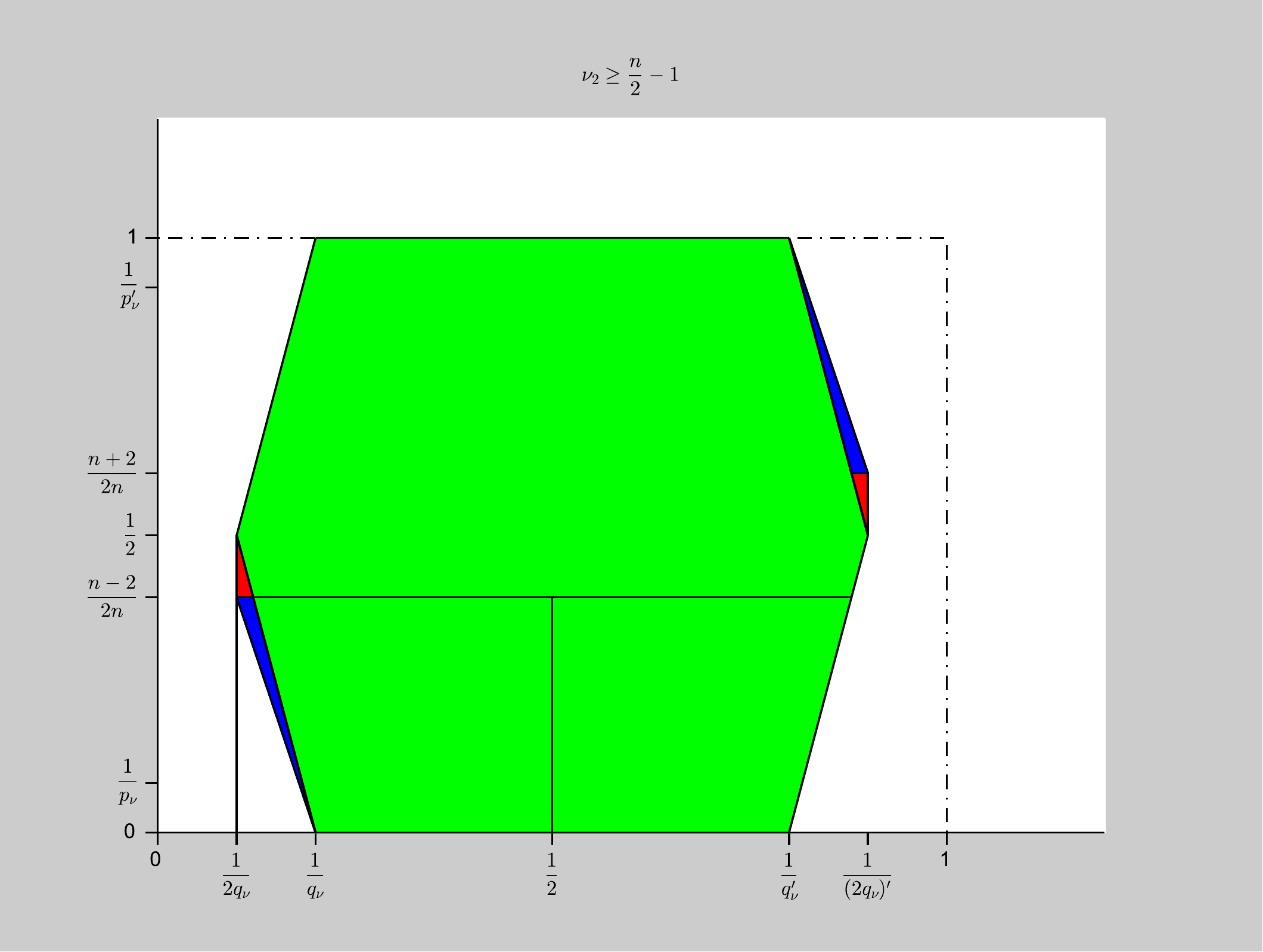}
	\caption{}
	\label{fig4}
\end{figure}

\vskip 2truemm

Figures \ref{fig1}, \ref{fig2}, \ref{fig3} and \ref{fig4} illustrate the regions of boundedness of the Bergman projector $P_\nu$ of $T_{\Lambda_n}.$

An application of Theorem 2.1 (the transference principle) then gives the following result (which improves \cite[Theorem 6.2.15]{Go} and \cite[Theorem 2.3]{Nana} for $D=D(\Gamma,\,F)$):

\begin{thm}
Let $\nu = (\nu_1, \nu_2)\in\mathbb{R}^2$ with $\nu_1>\frac 12$ and $\nu_2>1.$  The Bergman projector $P_\nu$ of the Pyateckii-Shapiro
domain $D(\Gamma,\,F)$ extends to a bounded operator on $L^{p, q}_\nu (D(\Gamma,\,F))$ for the following values of $p$ and $q:$
\begin{enumerate}
\item
$1\leq p \leq \infty$ and $\frac {2p_\sharp \nu_2}{2p_\sharp \nu_2 - 1}<q<2p_\sharp \nu_2;$
\item
$2\leq p \leq 6$ and $\frac {\nu_2}{\nu_2 -\frac 1{2p'}}<q<4\nu_2;$
\item
$p>6$ and $2<q<\frac {\nu_2}{\frac 3{2p'}-1};$

\item
the couples $(p, q)$ obtained by symmetry and complex interpolation from the previous couples.
\end{enumerate}
\end{thm}
The couples $(p, q)$ described in the previous theorem are represented in the figure depicted below.

\vskip 2truemm

\begin{figure}[tbph]
	\centering
	\includegraphics[width=0.7\linewidth]{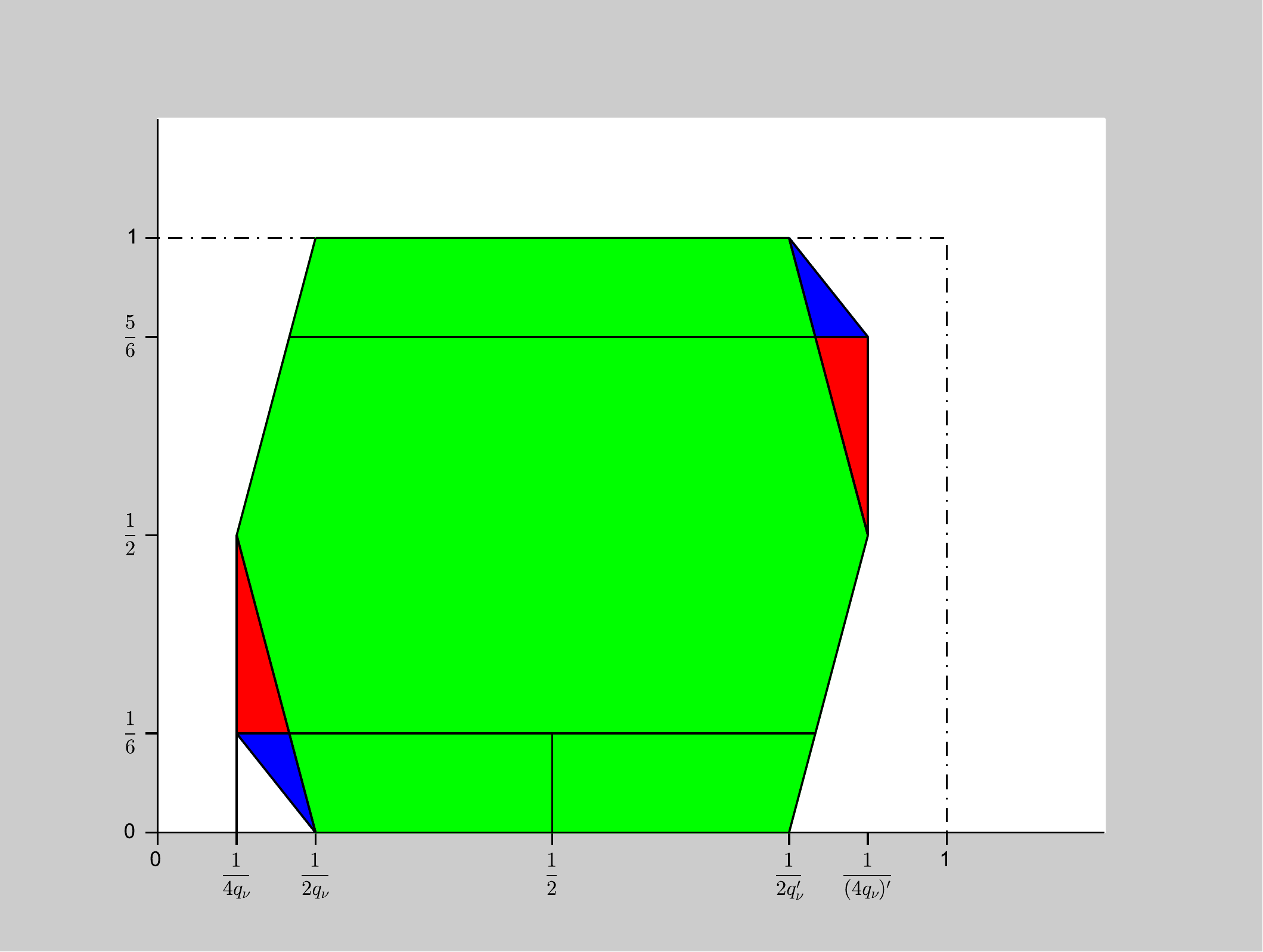}
	\caption{}
	\label{fig5}
\end{figure}

\vskip 2truemm

Further improvements will follow for homogeneous Siegel domains of type II associated to symmetric cones if 
one could solve the conjecture
stated in \cite{BBGRS} for tube domains over symmetric cones.

The plan of the sequel of the present paper is as follows. In section 3, we review the analysis on homogeneous
cones and on homogeneous Siegel domains of type II. Most of the results in this section are taken from \cite{Nana}.
In section 4, we restrict to tube domains over homogeneous cones and we introduce the action of the Box operator,
generalizing results from \cite{BBPR} for Lorentz cones. For these domains, we exhibit a necessary and sufficient
condition for the boundedness of the Bergman projector in terms of a Hardy type inequality for the Box operator. In section 5,
we prove the Hardy type inequality for homogeneous Siegel domains of type II and we apply it to prove the main result, namely Theorem 2.1.
 The proofs of Theorem 2.3 and Theorem 2.4 are given in section 6.

\section{Analysis on homogeneous cones and in homogeneous Siegel domains of type II}
Let $n\geq 3$ and $D$ be a homogeneous Siegel domain of type II associated to the homogeneous cone $\Omega$ and the $\Omega$-Hermitian form $F.$

\subsection{Basic results on homogeneous cones and in homogeneous Siegel domains of type II}
In this section, we recall the following results whose proofs are essentially in \cite{NT}.

\begin{lem}
\label{integ}\cite[Corollary 4.16]{NT} Let $\nu=(\nu_1,\ldots,\nu_r)\in\R^r$ such that $\nu_j>\frac{m_j}{2},\quad j=1,\ldots,r.$ Then
$$\int_\O e^{-(\xi|y)}Q^{\nu-\tau}(y)dy=\Ga_{\O}(\nu)(Q^*)^{-\nu}(\xi),\,\,\,\xi\in\O^*,$$
where $\Ga_\O(\nu)$ denotes the gamma integral \cite{NT} in the cone $\Omega.$
\end{lem}

\begin{remark}
{\rm It is well-known that the fundamental rational functions $Q_j$ (resp. $Q^*_j)$ can extended as a zero-free analytic function $Q_j (\frac zi)$ (resp. $Q^*_j (\frac zi))$ on the tube domain $V+i\O$ (resp. $V'+i\O^*)$. In particular, it follows from Lemma \ref{integ} that
if $\zeta\in V'+i\O^*$ and $\nu_j>\frac{m_j}{2},\quad j=1,\ldots,r,$ we set}
 \Bea \int_\O\label{int}
e^{i(\zeta|y)}Q^{\nu-\tau}(y)dy=\label{cor1}\Ga_{\O}(\nu)(Q^*)^{-\nu}\left(\frac{\zeta}{i}\right).
\Eea
\end{remark}

\begin{lem} \label{13} \cite[Lemma 4.19]{NT} Let $\mu=(\mu_1,\mu_2,\ldots,\mu_r)\in\R^r$ and
$\la=(\la_1,\la_2,\ldots,\la_r)\in\R^r.$ For all $y\in \O,$ the integral
$$J_{\mu\la}(y)=\int_\O Q^\mu(y+v) Q^{\la-\tau}(v) dv$$
is finite if and only if \Beas \la_j>\frac{m_j}{2}, \quad
\mu_j+\la_j<-\frac{n_j}{2},\quad \quad j=1,\ldots,r. \Eeas In this
case, there is a positive constant $M_{\la\mu}$ such that
$$J_{\mu\la}(y)=M_{\la\mu}Q^{\mu+\la}(y).$$
\end{lem}

\begin{lem} \label{14} \cite[Lemma 4.20]{NT} Let  $\a=(\a_1,\a_2,\ldots,\a_r)\in \R^r.$
\hskip 0.5cm
The integral \Bea\label{Ja} J_\a(y)=\int_V
\left|Q^{-\a}\left(\frac{x+iy}{i}\right)\right|dx\quad \quad
(y\in\O) \Eea converges if and only if
$\a_j>1+n_j+\frac{m_j}{2},\,\,j=1,\ldots,r.$ In this case, there is a positive constant $c_\a$ such that \Beas
J_\a(y)=c_\a Q^{-\a+\tau}(y). \Eeas

\end{lem}

The following two results were stated in \cite{Nana} .

\begin{cor}
$A^{p, q}_\nu (D)$ is a Banach space.
\end{cor}

\begin{lem}\label{density}
Assume that $\mu = (\mu_1,...,\mu_r)$  and $\nu = (\nu_1,...,\nu_r)$ belong to $\mathbb R^r$ and satisfy $\mu_j, \hskip 1truemm \nu_j > \frac {m_j + b_j}2, \hskip 2truemm j=1,...,r.$ The subspace $(A^{p, q}_\nu \cap A^{2}_\mu)(D)$ is dense in $A^{p, q}_\nu (D).$
\end{lem}

Let $\nu=(\nu_1,\nu_2,\ldots,\nu_r)\in\R^r$
such that $\nu_j>\frac{m_j+b_j}{2},\,\,j=1,\ldots,r.$ Following \cite{BT}, we shall
denote $L_{(-\nu)}^2(\O^*\times \C^m)$ the Hilbert space of functions $g:\O^*\times\C^m\to\C$ such that:
\begin{itemize}
\item[i)] for all compact subset $K_1$ of $\C^n$ contained in $\O^*$ and for all compact subset $K_2$ of $\C^m,$ the mapping $u\mapsto g(\cdot,u)$ is holomorphic on $K_2$ with values in $L^2(K_1,-\nu),$ where $$L^2(K_1,-\nu)=\{f:K_1\to\C:\int_{K_1}|f(\xi)|^2(Q^*)^{-\nu+\frac{b}{2}}(\xi)d\xi<\infty\};$$
    \item[ii)] the function $g\in L^2(\O^*\times\C^m,\,(Q^*)^{-\nu+\frac{b}{2}}(\xi) e^{-2(F(u,\,u)|\xi)}d\xi dv(u)).$
\end{itemize}
We then
define by
$$\mathcal L g(z,\,u)=(2\pi)^{-\frac{n}{2}}\int_{\O^*}e^{i(z|\xi)}g(\xi,\,u)d\xi$$
the "Laplace transform" of any function $g \in L_{(-\nu)}^2(\O\times\C^m).$ Now, we
recall the Plancherel-Gindikin result found in \cite[Theorem II.2]{BT} which is a generalization of the Paley-Wiener Theorem \cite[Theorem 5.1]{NT}.

\begin{thm}\label{16}
Let  $\nu=(\nu_1,\nu_2,\ldots,\nu_r)\in \R^r$ with
$\nu_j>\frac{m_j+b_j}{2},\,\,j=1,\ldots,r.$ A function $G$ belongs
to $A_\nu^2 (D)$ if and only if $G=\mathcal Lg$, with $g\in
L_{(-\nu)}^2(\O^*\times \C^m).$ Moreover there is a positive constant $e_{\nu,b}$ such that \Bea
 \|G\|_{A_\nu^2
(D)}=\label{PW1}e_{\nu,b}\|g\|_{L_{(-\nu)}^2(\O^*\times\C^m)}.\Eea
\end{thm}

\subsection{Integral operators associated to the Bergman projector}

Let $\mu,\a\in\R^r$ and $f$ be a bounded function with compact support on $D.$  We consider the integral operators:
\begin{equation}\label{operator}
 T_{\mu,\a}f(z,u)=Q^\a(\Im m\, z-F(u,u))\int_DB_{\mu+\a}((z,u),(w,t))f(w,t)dV_\mu(w,t)
\end{equation}
and
\begin{equation}
 T^+_{\mu,\a}f(z,u)=Q^\a(\Im m\, z-F(u,u))\int_D \left \vert B_{\mu+\a}((z,u),(w,t))\right \vert f(w,t)dV_\mu(w,t).
\end{equation}

\begin{thm}\label{essai}
Let $\a,\nu,\mu\in\R^r$ and $1\leq p,q\leq \infty.$ Then the operator $T^+_{\mu,\a}$ extends boundedly to $L^{p,q}_\nu(D)$ whenever for all $j=1,\ldots,r$
the parameters satisfy the following:
$$\mu_j+\a_j>\frac 12(m_j+n_j+b_j)$$ and
$$\mu_jq-\nu_j>(q-1)\left(\frac{m_j}{2}+\frac{b_j}{2}\right)+\frac{n_j}{2},\,\,\,\,\a_jq+\nu_j>\frac{m_j}{2}+\frac{b_j}{2}+(q-1)\frac{n_j}{2}.$$
\end{thm}

\begin{proof}

We follow the scheme of \cite[Proof of Theorem 2.1]{Nana}. Let  $\mu=(\mu_1,\mu_2,\ldots,\mu_r)\in\R^r$ be such that $\mu_j>\frac{m_j+b_j}{2},\,\,j=1,\ldots,r.$

We shall denote $U=\{(t,u): u\in\C^m, t\in\O+F(u,u)\}.$ We define the measure $\mathcal V_\mu, \hskip 1truemm \mu \in \mathbb R^r,$ on $U$ by
$$d\mathcal V_\mu (t, u) = Q^{\mu - \frac b2 -\tau} (t-F(u, u))dtdv(u)$$
and we define
$L_\mu^q(U)$ as the space of all $g:U\to\C$ with norm given by
\Beas\|g\|_{L_\mu^q(U)}^q&=&\int_{\C^m}\int_{\O+F(u,u)}|g(t,u)|^qd\mathcal V_\mu(t,u)\\&=&\int_{\C^m}\int_{\O}|g(y+F(u,u),u)|^qQ^{\mu-\frac{b}{2}-\tau}(y)dydv(u).\Eeas

 We will need the following result.

\begin{prop}\cite[Proposition 5.2]{Nana}\label{new}
Let $u,s\in\C^m;\,\,y\in\O+F(u,u)$ and $t\in\O+F(s,s).$ For $\la=(\la_1,\la_2,\cdots,\la_r)\in\R^r,$
the integral $$I_\la(y,u,t)=\int_{\C^m}Q^{-\la}(y+t+F(s,s)-2\Re e\,F(u,s))dv(s)$$
converges if $\la_j-b_j>\frac{n_j}{2},\,j=1,\ldots,r.$ In this case, there is a positive constant $C_\la$ such that
\Bea\label{new1}
I_\la(y,u,t)=C_\la Q^{-\la+b}(y-F(u,u)+t).
\Eea
\end{prop}

Next, we shall use the following notation: for all $u,s\in\C^m,$ $$A=\Re e\,F(u,s).$$

Also we shall denote
$$f_{y,u}(x)=f(x+iy,u).$$

Thus, for $f\in L_\mu^{p,q}(D),$ using Minkowski's inequality for integrals, Young's inequality and Lemma \ref{14}, we get
\Beas
\|T^+_{\mu,\a}f\|_{L_\nu^{p,q}(D)}\leq C\|R_{\mu,\a}g\|_{L^q_\nu(U)}
\Eeas
where for $s\in\C^m$ and $t\in\O+F(s,s),$
$$g(t,s)=\|f_{t,s}\|_p$$
and $R_{\mu,\a}$ is the integral operator with positive kernel defined on $L^q_\nu(U)$ by
\begin{equation}
 \label{T_g}  R_{\mu,\a}g(y,\,u)=Q^\a(y-F(u,u))\int_{\C^m}\int_{\O+F(s,\,s)} Q^{-\mu-\a-\frac{b}{2}}(y-2A+t)g(t,s)d\mathcal V_\mu(t,s).
  \end{equation}

Observe that $R_{\mu,\a}$ is a self-adjoint operator if $\a=0$ and $\mu=\nu.$ 

To prove  Theorem \ref{essai}, it suffices therefore to prove the boundedness of the operator $R_{\mu,\a}$ on $L^q_\nu(U).$
\begin{thm}\label{T}
Let $\mu=(\mu_1,\ldots,\mu_r)\in \R^r$ and $\a=(\a_1,\ldots,\a_r)\in \R^r$ such that
$\mu_j+\a_j>\frac{m_j+n_j+b_j}{2},\,\,j=1,\ldots,r.$  The operator $R_{\mu,\a}$ is
bounded on $L_\nu^q(U)$ whenever
$$\mu_jq-\nu_j>(q-1)\left(\frac{m_j}{2}+\frac{b_j}{2}\right)+\frac{n_j}{2},\,\,\,\,\a_jq+\nu_j>\frac{m_j}{2}+\frac{b_j}{2}+(q-1)\frac{n_j}{2}.$$


\end{thm}

\begin{proof}

 We will use Schur's Lemma (See \cite{FR}). The kernel of the operator
$R_{\mu,\a}$ relative to the measure $dV_\nu(t,s)$ is given by
$$N(y,u;t,s)= Q^\a(y-F(u,u))Q^{-\mu-\a-\frac{b}{2}}(y-2A+t)Q^{\mu-\nu}(t-F(s, s))$$
and it is positive. By Schur's Lemma, it is sufficient to find a
positive and measurable function  $\varphi$ defined on $U$ such
that
\begin{eqnarray}
\label{q} \int_{\C^m}\int_{\O+F(u,\,u)} N(y,u;t,s) \varphi(y,\,u)^{q}d\mathcal V_\nu(y,u)=
C\varphi(t,\,s)^{q} \Eea and \Bea \label{q'} \int_{\C^m}\int_{\O+F(s,\,s)} N(y,u;t,s)
\varphi(t,\,s)^{q'}d\mathcal V_\nu(t,s)= C\varphi(y,\,u)^{q'}.
\end{eqnarray}

We take as test functions $ \varphi(t,\,s)=Q^{\gamma}(t-F(s,\,s))$ where
$\gamma=(\gamma_1,\ldots,\gamma_r)\in\R^r$ has to be determined. The left-hand side of (\ref{q}) equals
$$K(t,s)=Q^{\mu-\nu}(t-F(s,s))\int_\O I_{-\mu-\a-\frac{b}{2}}(t,s,y)Q^{\ga q+\nu+\a-\frac{b}{2}-\tau}(y)dy.$$
Using (\ref{new1}), we get
\begin{equation}
 \label{new2}
K(t,s)=CQ^{\mu-\nu}(t-F(s,s))\int_\O Q^{-\mu-\a+\frac{b}{2}}(y+t-F(s,s))Q^{\ga q+\nu+\a-\frac{b}{2}-\tau}(y)dy.
\end{equation}

An
application of Lemma \ref{13}  gives that (\ref{new2}) holds whenever
\begin{eqnarray*}
\frac{-\nu_j-\a_j+\frac{m_j}{2}+\frac{b_j}{2}}{q}<\gamma_j<\frac{\mu_j-\nu_j-\frac{n_j}{2}}{q},\quad
j=1,\ldots,r.
\end{eqnarray*}

Likewise, (\ref{q'}) holds when
\begin{eqnarray*}
\frac{-\mu_j+\frac{m_j}{2}+\frac{b_j}{2}}{q'}<\gamma_j<\frac{\a_j-\frac{n_j}{2}}{q'},\quad
j=1,\ldots,r.
\end{eqnarray*}
For these intervals to be non-empty, we need $\mu_j+\a_j>\frac{m_j+n_j+b_j}{2},\,\,j=1,\ldots,r.$

 The identities (\ref{q}) and  (\ref{q'}) are
simultaneously satisfied if each $\gamma_j,\,\,j=1,\ldots,r,$ satisfy
the following  condition
\begin{equation}
 \label{cond2} \gamma_j\in
\left]\frac{-\nu_j-\a_j+\frac{m_j}{2}+\frac{b_j}{2}}{q'}, \frac{\mu_j-\nu_j-\frac{n_j}{2}}{q'}\right[\bigcap\left]\frac{-\mu_j+\frac{m_j}{2}+\frac{b_j}{2}}{q},\frac{\a_j-\frac{n_j}{2}}{q}\right[.
\end{equation}
 The intersection in (\ref{cond2}) is not empty if
$\frac{-\nu_j-\a_j+\frac{m_j}{2}+\frac{b_j}{2}}{q'}<\frac{\a_j-\frac{n_j}{2}}{q}$ and
$\frac{-\mu_j+\frac{m_j}{2}+\frac{b_j}{2}}{q}<\frac{\mu_j-\nu_j-\frac{n_j}{2}}{q'};$ that is
for any $j=1,\ldots,r,$

$\a_jq+\nu_j>(q-1)\frac{n_j}{2}+\frac{m_j}{2}+\frac{b_j}{2}$ and $\mu_jq-\nu_j>\frac{n_j}{2}+(q-1)\left(\frac{m_j}{2}+\frac{b_j}{2}\right).$
\end{proof}
\end{proof}


\begin{cor}\label{P+}
Let $\mu=(\mu_1,\mu_2,\ldots,\mu_r)\in\R^r$ and $\nu=(\nu_1,\nu_2,\ldots,\nu_r)\in\R^r$ such that
$\mu_j>\frac{m_j+n_j+b_j}{2},\,\,\,j=1,\ldots,r$ and $\nu_j>\frac{m_j+b_j}{2},\,\,\,j=1,\ldots,r.$ Assume that $1\leq p,q\leq \infty.$ Then $P_\mu^+$ is bounded on $L^{p,q}_\nu(D)$ whenever
for all $j=1,\ldots,r,$ we have
$$\frac{\nu_j-\frac{m_j}{2}-\frac{b_j}{2}+\frac{n_j}{2}}{\mu_j-\frac{m_j}{2}-\frac{b_j}{2}}<q<1+\frac{\nu_j-\frac{m_j}{2}-\frac{b_j}{2}}{\frac{n_j}{2}}.$$
In this case, the Bergman projector $P_\mu$ extends to a bounded operator from $L^{p,q}_\nu(D)$ onto $A^{p,q}_\nu(D).$
\end{cor}

\begin{proof}

Just take $\a=0$ in Theorem \ref{essai}. The case $q=1$ is left as an exercise: also cf. \cite{BT}.
\end{proof}

\begin{remark}\label{3.12}
Let $k$ be a positive integer, let $\rho \in \mathbb R^r$ be such that $\rho_j >0$ for every $j=1,\cdots, r.$
Let $1\leq p \leq \infty$ and $2<q<\infty.$ In view of Corollary \ref{P+}, the operator $P^+_{\nu+k\rho}$
(and hence the Bergman projector $P_{\nu+k\rho}$) is bounded on $L^{p, q}_{\nu+kq\rho}$ for $\nu=(\nu_1,\nu_2,\ldots,\nu_r)\in\R^r$ such that
$\nu_j>\frac{m_j+n_j+b_j}{2},\,\,\,j=1,\ldots,r$ and $k$ large.
\end{remark}

\begin{prop}
Let $\nu=(\nu_1,\nu_2,\ldots,\nu_r)\in\R^r$ be such that
$\nu_j>\frac{m_j+b_j}{2},\,\,\,j=1,\ldots,r.$ Assume $\mu=(\mu_1,\mu_2,\ldots,\mu_r)\in\R^r$ and $p,q$ are such that $P_\mu$ extends as a bounded operator
in $L^{p,q}_\nu(D).$ Then
\begin{itemize}
\item[(i)]
the Bergman space $A^{p,q}_\nu(D)$ is the closed linear span of the set $\{B_\mu(\cdot,(w,t)),\,(w,t)\in D\}.$ In particular,
$B_\mu(\cdot,(w,t))\in L^{p,q}_\nu(D).$
\item[(ii)]
$P_\mu$ is the identity on $A^{p,q}_\nu(D);$ in particular, $P_\mu (L^{p,q}_\nu(D))=A^{p,q}_\nu(D).$
\end{itemize}
\end{prop}

\begin{proof}
\begin{itemize}
\item[(i)]
We start by establishing that $B_\mu(\cdot,(w,t))\in A^{p,q}_\nu(D)$ for all $(w,t)\in D.$ Let $f,g\in \mathcal C_c (D)$ (continuous functions on $D$ with compact support). Then
\Beas
\langle P_\mu g,f\rangle_\nu=\int_D P_\mu g(z,u)\overline{f(z,u)}dV_\nu(z,u)=\langle g,P^*_\mu f\rangle_\nu
\Eeas
where
\begin{equation} \label{Pstar}
P^*_\mu f(z,u)=Q^{\mu-\nu}(\Im m\,z-F(u,u))\int_D B_\mu((z,u),(w,t))f(w,t)dV_\nu(w,t)
\end{equation}
i.e.
\begin{equation}\label{Pstarplus}
P^*_\mu f=T_{\nu,\mu-\nu}f
\end{equation}
according to (\ref{operator}). By density of $\mathcal{C}_c(D)$ in $L^{p',q'}_\nu(D)$ and continuity of $P^*_\mu$ on $L^{p',q'}_\nu(D),$ we conclude that
\begin{equation}
 P^*_\mu f=T_{\nu,\mu-\nu}f,\,\,\,\forall f\in L^{p',q'}_\nu(D).
\end{equation}
Now, the boundedness of $P_\mu^*$ on $L^{p',q'}_\nu(D)$ means boundedness of $T_{\nu,\mu-\nu}$ on $L^{p',q'}_\nu(D).$
Therefore, $B_\mu(\cdot,(w,t))\in L^{p,q}_\nu(D)$ for all $(w,t)\in D.$

To conclude, we show that for $f\in L^{p',q'}_\nu(D)$ such that
\begin{equation}\label{duality}
 \langle f,B_\mu(\cdot,(w,t))\rangle_\nu=0,\,\,\,\forall (w,t)\in D,
 \end{equation}
we also have $\langle f,F\rangle_\nu=0$ for all $F$ in a dense subspace of $A^{p,q}_\nu(D).$ Now, identities (\ref{duality}), (\ref{Pstar})
and (\ref{Pstarplus}) imply that
$$T_{\nu,\mu-\nu}f=0.$$
Thus, for all $F\in A^{p,q}_\nu(D)\cap A^{2,2}_\mu(D),$ we have
$$\langle f,F\rangle_\nu=\langle f,P_\mu F\rangle_\nu=\langle P_\mu^*f,F\rangle_\nu=\langle T_{\nu,\mu-\nu}f,F\rangle_\nu=0.$$
\item[(ii)]
To prove $(ii),$ we just observe that $P_\mu$ is the identity on the subspace $(A^{p, q}_\nu \cap A^{2, 2}_\mu) (D)$ which is dense on the space $A^{p, q}_\nu (D)$ by Lemma \ref{density}.
\end{itemize}
\end{proof}

\section{The action of the Box operator in tube domains over convex homogeneous cones}

In this section, we restrict to the case where the $\Omega-$Hermitian form $F$ is identically zero and $m=0$. We denote $T_\Omega$ the tube domain over the homogeneous cone $\Omega$ of rank $r,$ i.e. $T_\Omega = V+i\Omega.$ For $\nu \in \mathbb R^r, \hskip 2truemm 1\leq p \leq \infty$ and $1\leq q < \infty,$ let $L^{p, q}_\nu (T_\Omega)$ be the space of measurable functions on $T_\Omega$ such that
$$||f||_{L^{p, q}_\nu (T_\Omega)} :=  \left(\int_{\Omega} \left(\int_V |f(x+iy)|^p dx\right)^{\frac qp} Q^{\nu - \tau} (y)dy\right)^{\frac 1q}$$
is finite (with obvious modification if $p=\infty)$.

We call $A^{p, q}_\nu (T_\Omega)$ the subspace of $L^{p, q}_\nu (T_\Omega)$ consisting of holomorphic functions. For $A^{p, q}_\nu (T_\Omega)\neq \{O\},$ we assume in the sequel that $\nu = (\nu_1,\cdots,\nu_r)$ is such that $\nu_j > \frac {m_j}2$ for each $j=1, \cdots, r.$

We  call $\mathbb P_\nu$ the weighted Bergman projector on $T_\Omega$ and $\mathbb B_\nu$ the associated weighted Bergman kernel on $T_\Omega.$
\begin{defn}
Let $\rho\in \mathbb R^r$ be such that $\rho_j > 0$ for every $j=1,\cdots,r.$
We say that $\rho$ is an $\Omega-$integral vector if the fundamental compound function $(Q^*)^\rho (\xi)$ is a polynomial in $\xi.$
\end{defn}

In the sequel, we fix an $\Omega$-integral vector $\rho.$  We shall now adapt some proofs of \cite[section 6]{BBPR} to tube domains over open convex homogeneous  cones.

\begin{defn} The generalized wave operator (the Box) $\Box=\Box_x$ on the cone $\Omega$ is the differential operator  defined by the equality
$$\Box_x [e^{i(x|\xi)}] = (Q^*)^\rho (\xi)e^{i(x| \xi)} \hskip 2truemm {\rm where} \hskip 2truemm  \xi \in \mathbb R^n.$$
When applied to a holomorphic function on the tube domain $T_\Omega$ over the cone $\Omega,$ we have $\Box = \Box_z = \Box_x$ where $z=x+iy.$ In view of Remark 3.2, for every $\mu \in \mathbb R^r$ such that $\mu_j >  \frac {m_j}2, \hskip 2truemm j=1,\cdots, r,$ we have
\begin{equation}\label{box}
\Box \left(Q^{-\mu} \left(\frac zi\right)\right) = c_\mu Q^{-\mu -\rho} \left(\frac zi\right), \quad z\in T_\Omega.
\end{equation}
\end{defn}

\begin{lem}
For every $f\in A^{2, 2}_\nu (T_\Omega)$ and every $h\in H,$ if we set $f_h (z) = f(hz),$ then
$$\Box (f_h) = Q^\rho(h\cdot e)(\Box f)_h.$$
\end{lem}

\begin{proof}
In view of the previous theorem, we have
$$f_h (z) = (2\pi)^{-\frac n2} \int_{\Omega^*} e^{i(\pi[h] z|\xi)}g(\xi)d\xi = (2\pi)^{-\frac n2} \int_{\Omega^*} e^{i(z|\pi[h^\star]\xi)}g(\xi)d\xi$$
for some $g\in L^2_{(-\nu)} (\Omega^*).$ The definition of $\Box$ and identity (6) give
\begin{eqnarray*}
\Box (f_h) (z)  &=& (2\pi)^{-\frac n2} \int_{\Omega^*} (Q^*)^\rho (\pi[h^\star]\xi)e^{i(z|h^\star \cdot \xi)}g(\xi)d\xi\\
&=&(2\pi)^{-\frac n2} (Q^*)^\rho (h^\star \cdot e)\int_{\Omega^*} (Q^*)^\rho (\xi)e^{i(\pi [h]|\xi)}g(\xi)d\xi\\
&=&Q^\rho (h\cdot e)(\Box f)_h (z)
\end{eqnarray*}
since by (3), $Q^* (h^\star \cdot e)=Q(h\cdot e).$
\end{proof}

\begin{prop}\label{Boxee}
The Box operator $\Box^k$ is bounded from $A^{p, q}_\nu (T_\Omega)$ to $A^{p, q}_{\nu + kq\rho } (T_\Omega)$ for all $1\leq p, q\leq \infty$ and $\nu=(\nu_1,\cdots,\nu_r)\in \mathbb R^r$ such that $\nu_j > \frac {m_j}2, \hskip 2truemm j=1,...,r.$
\end{prop}

\begin{proof} It suffices to prove the proposition for $k=1.$ The other cases are obtained by induction on $k.$ We denote by $d$ the invariant distance on the cone $\Omega.$ The Cauchy integral formula for derivatives implies that, if $f$ is holomorphic,
$$|\Box f (x+i\mathbf e)| \leq C\int_{|\xi|<1, \hskip 1truemm d(\eta, {\mathbf e})<1 } |f(x-\xi+i\eta)|d\xi d\eta.$$
Hence by the Minkowski integral inequality,
$$||\Box f (\cdot +i\mathbf e)||_p \leq C\int_{d(\eta, {\mathbf e})<1} ||f(\cdot +i\eta)||_p Q^{-\tau} (\eta)d\eta.$$
Here we have introduced the $H$-invariant measure on the cone $\Omega:$
$$dm(\eta) = Q^{-\tau} (\eta)d\eta.$$
We recall thet $Q(\eta) \sim Q(y)$ if $d(\eta, y) <1.$ (See for instance \cite[section 4]{NT}).
Let $f$ be in the dense subspace $(A^{p, q}_\nu  \cap A^{2}_\nu) (T_\Omega)$ and let $y=h\cdot e$ with $h\in H.$
A change of variables combined with the previous lemma gives that
$$||\Box f (\cdot +iy)||_p \leq CQ(y)^{-\rho}\int_{d(\eta, y)<1} ||f(\cdot +i\eta)||_p dm (\eta).$$
Then
\begin{eqnarray*}
||\Box f||^q_{A^{p, q}_{\nu + q\rho} (T_\Omega)} &\leq&C\int_\Omega \left(\int_{d(\eta, y)<1} ||f(\cdot +i\eta)||_p dm (\eta)\right)^q Q^{\nu - \tau} (y)dy\\
&\leq&C\int_\Omega \left(\int_{d(\eta, y)<1} ||f(\cdot +i\eta)||_p^q dm(\eta)\right)\left(\int_{d(\eta, y)<1} dm(\eta)\right)^{q-1}Q^{\nu - \tau} (y)dy
\end{eqnarray*}
by the H\"older inequality. Since
\begin{equation}\label{inv}
\int_{d(\eta, y)<1} dm(\eta)=\int_{d(\eta, \textbf e)<1} dm(\eta)=const.,
\end{equation}
we obtain
\begin{eqnarray*}
||\Box f||^q_{A^{p, q}_{\nu + q\rho} (T_\Omega)} &\leq&C\int_\Omega \left(\int_{d(\eta, y)<1} ||f(\cdot +i\eta)||_p^q Q^{- \tau} (\eta)d\eta\right)Q^{\nu-\tau} (y)dy\\
&\leq&C'\int_\Omega  \left(\int_{d(\eta, y)<1} ||f(\cdot +i\eta)||_p^q Q^{\nu-\tau}(\eta)d\eta\right)Q^{-\tau} (y)dy=C||f||_{A^{p, q}_\nu (T_\Omega)}.
\end{eqnarray*}
since $Q_j (y) \sim Q_j (\eta)$ when $d(\eta, y) <1.$ An application of the Fubini-Tonelli Theorem and of identity (\ref{inv}) gives that
\begin{eqnarray*}
||\Box f||^q_{A^{p, q}_{\nu + q} (T_\Omega)} &\leq&C\int_\Omega \left(\int_{d(\eta, y)<1} Q^{- \tau} (y)dy\right)||f(\cdot +i\eta)||_p^q Q^{\nu - \tau} (\eta)d\eta\\
&\leq&C'\int_\Omega  ||f(\cdot +i\eta)||_p^q Q^{\nu-\tau}(\eta)d\eta=C'||f||_{A^{p, q}_{\nu} (T_\Omega)}.
\end{eqnarray*}
\end{proof}

We show next the relations between properties of $\Box$ and boundedness of $\mathbb P_\nu.$ They will give us a necessary condition for the boundedness of $\mathbb P_\nu.$\\
Our considerations are based on the identity stated in the next lemma. We set
$$M_k f(x+iy) = Q^{k\rho} (y)f(x+iy).$$

\begin{lem}
Let $f$ be a continuous function with compact support in $T_\Omega.$ Then, for every positive integer $k,$
\begin{equation}\label{multiplication}
\Box^k (\mathbb P_\nu f) = \gamma_{\nu, k} \mathbb P_{\nu+k\rho} (M_{-k} f),
\end{equation}
where $\gamma_{\nu, k}$ is a non-zero constant and $\Box^k = \Box \circ ... \circ \Box \quad k$ times.
\end{lem}

\begin{prop}\label{Boxe}
Assume that $\mathbb P_\nu$ is bounded on $L^{p, q}_\nu (T_\Omega).$ Then  for every positive integer $k, \hskip 2truemm \Box^k$ is an
isomorphism of $A^{p, q}_\nu (T_\Omega)$ onto $A^{p, q}_{\nu +kq\rho} (T_\Omega).$ Moreover, for $f\in A^{p, q}_\nu (T_\Omega),$
\begin{equation}\label{diagram}
\mathbb P_\nu \circ M_k (\Box^k f) = \gamma_{\nu, k}f.
\end{equation}
\end{prop}

\begin{proof}
Clearly, $M_{-k}$ is an isometric isomorphism of $L^{p, q}_\nu (T_\Omega)$ onto $L^{p, q}_{\nu +kq\rho} (T_\Omega)$ with
$(M_k)^{-1}=M_{-k}.$ Let $f$ be a continuous function with compact support in $T_\Omega.$ By (\ref{multiplication}),
$$\mathbb P_{\nu+k\rho} f = \gamma_{\nu, k}^{-1} \Box^k \circ \mathbb P_\nu \circ M_k f.$$
By density, $P_{\nu+k\rho}$ is bounded on $L^{p, q}_{\nu +kq\rho} (T_\Omega).$ We then have the following commutative diagram.
\begin{eqnarray*}
L^{p, q}_\nu (T_\Omega) &\stackrel{\mathbb P_\nu}\longrightarrow&A^{p, q}_\nu (T_\Omega)\\
M_{-k} \downarrow &&\downarrow \gamma_{\nu, k}^{-1} \Box^k \\
L^{p, q}_{\nu +kq\rho} (T_\Omega) &\stackrel{\mathbb P_{\nu+k\rho}}\longrightarrow& A^{p, q}_{\nu +kq} (T_\Omega)
\end{eqnarray*}
where each map is continuous. By Remark 3.12 and Proposition 3.13, $\mathbb P_{\nu + k}$ is onto; then also $\Box^k$ is onto.

The rest of the proof goes in the following order. We first prove that $\Box^k$ is one to one for $p=q=2$ (in which case $P_\nu$
is obviously bounded). Then we prove (\ref{diagram}) and finally we show that $\Box^k$ is one to one for general $p, q.$

Let $f=\mathcal L g \in A^{2}_\nu (T_\Omega).$ Then
$$\Box^k f = C_k \mathcal L (Q^{k\rho} g)\in A^{2}_{\nu +2k\rho} (T_\Omega).$$
By Theorem 3.7, if $\Box^k f =0,$ then $g=0$ a.e. and hence $f=0.$

In order to prove (\ref{diagram}), it suffices to take $f \in (A^{p, q}_\nu \cap A^{2}_\nu) (T_\Omega),$
since the left-hand side of (\ref{diagram}) involves continuous operators.

Calling $G$ the left-hand side of (\ref{diagram}), then $G\in (A^{p, q}_\nu \cap A^{2, 2}_\nu)(T_\Omega)$
and, by the commutativity of the diagram above,
$$\Box^k G = \gamma_{\nu, k} \mathbb P_{\nu + k\rho} (\Box^k f) = \gamma_{\nu, k} \Box^k f.$$
By the injectivity of $\Box^k$ on $A^{2, 2}_\nu (T_\Omega),$ we conclude that $G=\gamma_{\nu, k} f.$

Finally assume that $f\in A^{p, q}_\nu (T_\Omega).$ Then the assumption $\Box^k f =0$ implies that $f=0$ by (\ref{diagram}).
\end{proof}

The following theorem was proved in \cite [Theorem 6.8]{NT}.

\begin{thm}\label{boundedness}
Let $\nu =(\nu_1,\cdots,\nu_r) \in \mathbb R^r$ such that $\nu_j > \frac {m_j+n_j+b_j}2$ for each $j=1,\cdots,r.$ We set $q_\nu := 1+\min \limits_{1\leq j \leq r} \frac {\nu_j - \frac {m_j}2}{\frac {n_j}2}.$ Let $1\leq p\leq \infty$ and $2\leq q<\infty.$
Then the weighted Bergman projector $\mathbb P_\nu$ is bounded from $L^{p, q}_\nu (T_\Omega)$ to $A^{p, q}_\nu (T_\Omega)$ for\\
$
\left \{
\begin{array}{clcr}
0&\leq \frac 1p &\leq \frac 12\\
\frac 1{q_\nu p'}&<\frac 1q &< 1-\frac 1{q_\nu p'}
\end{array}
\right.
$
\quad \quad or \quad \quad
$
\left \{
\begin{array}{clcr}
\frac 12&\leq \frac 1p &\leq 1\\
\frac 1{q_\nu p}&<\frac 1q &< 1-\frac 1{q_\nu p}
\end{array}
\right
.
.
$
\end{thm}

The following statement is a direct consequence of Proposition \ref{Boxe} and Theorem \ref{boundedness}.

\begin{cor}\label{Box}
For the values of $p, q$ given in Theorem 4.7, $\Box^k$ is an isomorphism of $A^{p, q}_\nu (T_\Omega)$ onto $A^{p, q}_{\nu +k\rho q} (T_\Omega)$
for every positive integer $k.$
\end{cor}

A partial converse of Corollary \ref{Box} also holds. (See \cite{BBGNPR, BBGRS} for tube domains over symmetric cones and \cite{BBPR}
for tube domains over Lorentz cones).

\begin{thm}
Assume that $1\leq p<\infty$ and $2\leq q<\infty$ and that for some $k\geq k_0,$ the inequality
$$||f||_{A^{p, q}_\nu (T_\Omega)} \leq C||\Box^k f||_{A^{p, q}_{\nu+k\rho q} (T_\Omega)}$$
holds. Then $\mathbb P_\nu$ is bounded on $L^{p, q}_\nu (T_\Omega).$
\end{thm}


\section{Proof of Theorem 2.1}
In this section, we start by proving the Hardy-type inequality for the homogeneous Siegel domain $D$ of type II. The notations are those of section 2.

\begin{thm}
Let $\nu \in \mathbb R^r$ be such that $\nu_i > \frac {m_i + n_i + b_i}2, \hskip 2truemm i=1,...,r.$ Let $1\leq p<\infty$ and $2\leq q<\infty.$
Assume that there exists a positive integer $k$ and a positive constant $C=C(k, p, q, \nu)$ such that for all $f\in A^{p, q}_\nu (D),$ the following Hardy type inequality holds.
\begin{equation}\label{hardy}
\int_{\mathbb C^n} \int_{\Omega + F(u, u)} \left(\int_V |f(x+iy, u)|^p dx\right)^{\frac qp} Q^{\nu -\tau - \frac b2} (y-F(u, u))dy dv(u)
\end{equation}
$$
\leq C\int_{\mathbb C^n} \int_{\Omega + F(u, u)} \left(\int_V |\Box^k_x f(x+iy, u)|^p dx\right)^{\frac qp} Q^{\nu +kq\rho-\tau - \frac b2} (y-F(u, u))dy dv(u).
$$
Then the Bergman projector $P_\nu$ of $D$ admits a bounded extension to $L^{p, q}_\nu (D).$
\end{thm}

\begin{proof}
We adapt the proof of \cite[Theorem 1.3]{BBGRS},  for tube domains over symmetric cones. In this reference, $\nu$ is a real number. We want to prove the existence of some constant $C$ such that, for $f\in (L^{p, q}_\nu \cap L^{2, 2}_\nu)(D),$  we have the inequality
$$||P_\nu f||_{A^{p, q}_\nu (D)} \leq C||f||_{ L^{p, q}_\nu (D)}.$$
Consider such an $f$ with $||f||_{ L^{p, q}_\nu (D)} = 1.$ Call $G=P_\nu f.$ By Fatou's Lemma, it is sufficient to prove that
the functions $G_\epsilon (z, u) := G(z+i\epsilon \mathbf e, u),$ which belong to $A^{p, q}_\nu (D),$ have norms uniformly bounded.
So using (\ref{hardy}), it is sufficient to show that $\Box_z^k G_\epsilon$ is uniformly in $L^{p, q}_{\nu + kq\rho} (D).$
To prove this, we apply (\ref{box}) to obtain the identity
$$\Box_z^k G_\epsilon (z, u) = C\int_D B_{\nu+k\rho} ((z+i\epsilon \textbf e, u), (w, t))f(w, t)dV_\nu (w, t).$$
In view of Proposition \ref{Boxee}, it suffices to prove that $P_{\nu +k\rho}$ is bounded on $L^{p, q}_{\nu + kq\rho} (D)$ for $k$ large.
We apply Remark \ref{3.12} to conclude.
\end{proof}

By Theorem 5.1, it suffices to prove the existence of a positive integer $k$ such that the Hardy type inequality (\ref{hardy}) is valid.
If we make the change of variable $y' = y-F(u, u),$ this inequality takes the following form.
$$
\int_{\mathbb C^n} \left(\int_{\Omega} \left(\int_V |f(x+iy' +iF(u, u), u)|^p dx\right)^{\frac qp} Q^{\nu -\tau - \frac b2} (y')dy'\right)dv(u)
$$
$$
\leq C\int_{\mathbb C^n} \left(\int_{\Omega} \left(\int_V |\Box^m_x f(x+iy'+iF(u,u), u)|^p dx\right)^{\frac qp} Q^{\nu +kq\rho-\tau - \frac b2} (y')dy'\right)dv(u).
$$
For any fixed $u\in \mathbb C^n,$ we consider the holomorphic function $f_u: T_\Omega \rightarrow \mathbb C$ defined by
$$f_u (x+iy') = f(x+iy'+iF(u, u), u).$$
Also observe that the property $f\in A^{p, q}_\nu (D)$ can be expressed in the following form.
$$||f||^q_{A^{p, q}_\nu (D)} = \int_{\mathbb C^n} \left(\int_{\Omega} \left(\int_V |f(x+iy' +iF(u, u), u)|^p dx\right)^{\frac qp} Q^{\nu -\tau - \frac b2} (y')dy'\right) dv(u) < \infty.$$
An application of the Fubini Theorem gives that, for almost all $u\in \mathbb C^n,$ the function $f_u$ belongs to $A^{p, q}_\nu (T_\Omega).$

Assume that the Bergman projector $\mathbb P_{\nu - \frac b2}$ of $T_\Omega$ is bounded on $A^{p, q}_{\nu -\frac b2} (T_\Omega).$
By Proposition \ref{Boxe}, this implies that for every positive integer $k,$ $\Box^k$ is an isomorphism of $A^{p, q}_{\nu -\frac b2} (T_\Omega)$
onto $A^{p, q}_{\nu +kq\rho-\frac b2} (T_\Omega).$ So there exists a positive constant $C$ such that
$$\int_{\Omega} \left(\int_V |f_u(x+iy)|^p dx\right)^{\frac qp} Q^{\nu -\tau - \frac b2} (y)dy$$
$$ \leq C\int_{\Omega} \left(\int_V |\Box^k_x f_u(x+iy)|^p dx\right)^{\frac qp} Q^{\nu +kq\rho -\tau - \frac b2} (y)dy$$
for almost all $u\in \mathbb C^n.$ An integration with respect to $u$ finishes the proof.

\section{The case of the Pyateckii-Shapiro Siegel domain of type II}
\subsection{Bergman projections and Besov-type spaces in tube domains over symmetric cones}
Let $\O$ be a symmetric cone of rank $r$ in a Euclidean Jordan algebra $V.$ We describe a  Littlewood-Paley decomposition adapted to to the geometry of $\O.$ Referring to \cite{BBGR}, we call $d$ the invariant distance in $\O.$ Let $\{\xi_j\}$ be a fixed $(\frac 12, 2)$-lattice in $\Omega$ and let $B_j$ be the d-ball $B_1 (\xi_j)$ with centre $\xi_j$ and radius 1. These balls $\{B_j\}$ form a covering of $\O.$ We choose a real function $\varphi_0\in C_c^\infty (B_2 (\mathbf e))$ such that
$$0\leq \varphi_0 \leq 1, \quad {\rm and} \quad \varphi_0\vert_{B_1 (\mathbf {e})} \equiv 1.$$
We write $\xi_j =g_j \mathbf e,$ for some $g_j\in T.$ Then, we can define $\varphi_j (\xi) = \varphi_0 (g_j^{-1} \xi),$ so that
$$\varphi_j \in C_c^\infty (B_2 ( {\xi_j})), \quad 0\leq \varphi_j \leq 1, \quad {\rm and} \quad \varphi_j\vert_{B_j} \equiv 1.$$
We assume that $\xi_0 = \bf e$ to avoid ambiguity of notation. By the finite intersection property of the lattice $\{\xi_j\},$ there exists a constant $c>0$ such that
$$\frac 1c \leq \Phi (\xi) := \sum \limits_j \varphi_j (\xi) \leq c.$$
We define the function $\psi_j$ by $\widehat \psi_j=\frac {\varphi_j}\Phi.$
The Besov-type spaces, $B_\nu^{p,q}, \hskip 2truemm \nu=(\nu_1,\ldots,\nu_r) \in \mathbb R^r, \hskip 1truemm 1\leq p \leq \infty, \hskip 1truemm 1\leq q<\infty,$ adapted to this
 Littlewood-Paley decomposition  are defined as the equivalence classes of tempered distributions which have
 finite seminorms
 \Bea\label{besov}
 \|f\|_{B_\nu^{p,q}}=\left[\sum_j(Q^*)^{-\nu}(\xi_j)\|f*\psi_j\|_p^q\right]^{\frac{1}{q}}.
 \Eea
 When $n=1,$ and $\xi_j=2^j,$ the norm (\ref{besov}) corresponds to the classical Besov space $B_{p,q}^{-\nu/q}(\R).$
We denote by $S_{\O}$ the space   of Schwartz functions
$f: V\rightarrow \mathbb C$ with $\mbox{Supp}\hskip 1truemm \widehat{f}\subset \overline{\O}.$ One basic tool is a special decomposition for functions in $S_{\O},$
\Bea\label{Littlewood-Paley}
f=\sum_jf*\psi_j,\,\,\,\mbox{for all}\,\,f\in S_{\O}.
\Eea
Moreover, we call $\mathcal D_{\Omega}$ the subspace of $S_{\O}$ consisting of those functions whose support is compact in $\Omega.$   We point out that the subspace $\mathcal D_{\Omega}$ is dense in $B_\nu^{p,q} .$




We further refer to \cite{DD}, \cite{DD1} and \cite{BBGR}.
We normalize the Fourier transform by
$$\widehat{f}(\xi)=\F f(\xi)=\frac{1}{(2\pi)^n}\int_Ve^{-i(x|\xi)}f(x)dx,\,\,\,\mbox{for}\,\,\xi\in V$$
and like in section 3, we  define the Laplace transform $\mathcal L$ by
$$\mathcal L (f)(z) =\frac{1}{(2\pi)^{\frac n2}}\int_\O e^{i(z|\xi)}f(\xi)d\xi,\,\,\,\mbox{for}\,\,z\in T_\Omega.$$
We call $\mathcal C$ the operator $C(f) = \mathcal L(\widehat f).$ We call $p'$  the conjugate index of $p$ and for $\nu =(\nu_1,\cdots,\nu_r)\in \mathbb R^r,$ we adopt the following notations.
$$p_\sharp = \min (p, p');$$
$$p_\nu = 1+\min \limits_{j=1,\cdots,r} \frac {\nu_j +\frac nr}{(\frac {n_j}2 -\nu_j)_+};$$
$$q_\nu = 1+\min \limits_{j=1,\cdots,r} \frac {\nu_j -\frac {m_j}2}{\frac {n_j}2};$$
$$q_\nu (p) = p_\sharp q_\nu;$$
$$\widetilde q_{\nu, p} = \min \limits_{j=1,\cdots,r} \frac {\nu_j +\frac {n_j}2}{\left(\frac {n_j}{2p'} -\left(1+\frac {m_j}2\right)\right)_+}.$$
We record the following theorem due to D. Debertol. (See for instance
\cite[Theorem 4.6]{DD}, \cite[Theorem 1.2 and Corollary 4.7]{DD1}).

\begin{thm}
Let $\nu =(\nu_1,\cdots,\nu_r)\in \mathbb R^r,$  such that $\nu_j >\frac{m_j}{2},\,\,j=1,\ldots,r$ and let $1<p<\infty$ and $1<q<\widetilde q_{\nu, p}.$ Then, for every $F\in A^{p,q}_\nu (T_\O)$, there is a tempered distribution $f$ in $V$ such that $||f||_{B^{p,q}_\nu } <\infty$ and $F=\mathcal C f.$ Moreover we have
\begin{enumerate}
\item
$\lim \limits_{y\rightarrow 0, \hskip 1truemm y\in \Omega} F(\cdot +iy) = f$ both in ${\mathcal S}'(V)$ and in $B^{p,q}_\nu;$
\item
$||f||_{B^{p,q}_\nu} \lesssim ||F||_{A^{p,q}_\nu}.$
\end{enumerate}
\end{thm}

D. Debertol also proved the following theorem found in \cite[Theorem 5.8]{DD} and  \cite[Theorem 1.3]{DD1}.

\begin{thm}
Let $\nu=(\nu_1,\cdots,\nu_r)\in \mathbb R^r$ such that $\nu_j >\frac {m_j}2, \hskip 2truemm j=1,\cdots,r$ and $1<p<p_\nu, \hskip 2truemm q'_\nu (p) < q < \widetilde q_{\nu, p}.$ The following assertions are equivalent.
\begin{enumerate}
\item
The Bergman projector $\mathbb P_\nu$ of $T_\O$ admits a bounded extension from $L^{p, q}_\nu (T_\Omega)$ to $A^{p, q}_\nu (T_\Omega).$
\item
The operator $\mathcal C$ is an isomorphism from $B^{p, q}_\nu $ to $A^{p, q}_\nu (T_\Omega).$
\end{enumerate}
\end{thm}

The following two results are generalizations of  \cite[Theorem 4.11]{BBGR} and \cite[Lemma 4.14]{BBGR}. For the analysis on symmetric cones, we also refer to \cite{FaKo}.

 \begin{thm}\label{4.11}
 Let $\nu=(\nu_1,\ldots,\nu_r)\in\R^r$ such that $\nu_j>\frac{m_j}{2},\,\,j=1,\ldots,r$ and $1\leq p,\,s<\infty.$ Assume that there exist a number
 $\da>0,$ a vector $\mu=(\mu_1,\ldots,\mu_r)\in\R^r$ with $\mu_j >0, =1,\ldots,r$ and a constant $C=C(\mu,\da)>0$ such that the estimate
 \Bea\label{4.12}
 \left\|\sum_j f_j\right\|_p\leq C\left[\sum_j(Q^*)^{-\mu}(\xi_j)e^{\da(\xi_j|{\bf e})}\|f_j\|^s_p\right]^\frac{1}{s}
 \Eea
 holds for every finite sequence $\{f_j\}\subset L^p(V)$ with $\mbox{Supp}\hskip 1truemm \widehat{f_j}\subset B_2 (\xi_j).$
We assume that the index $q$ satisfies one of the following conditions.
\begin{enumerate}
\item[(i)]
$1\leq q\leq s$ and $q < s\min \limits_{j=1,\ldots,r}\frac{\nu_j-\frac{m_j}{2}}{\mu_j};$
\item[(ii)]
$s<q<\min \left \{s\min\limits_{j=1,\ldots,r}\frac{\nu_j-\frac{m_j}{2}+\frac{n_j}{2}}{\mu_j+\frac{n_j}{2}}, \widetilde{q}_{\nu,p}\right\}.$
\end{enumerate}
Then for every function $f\in \mathcal S_{\O},$  the function $F=\mathcal{C}(f)$ belongs to $A_\nu^{p,q}(T_\O),$
 and moreover,
 $$\|F\|_{A_\nu^{p,q}(T_\O)}\lesssim\|f\|_{B_\nu^{p,q}(T_\O)}.$$
 \end{thm}

 \begin{lem}\label{4.14}
 Let $1\leq p,\,s<\infty$ and assume that {\rm (\ref{4.12})} holds for some number $\da>0$ and some vector $\mu=(\mu_1,\ldots,\mu_r)\in\R^r.$    Then for every
 $f\in\mathcal{D}_{\O}$ and $y\in\O,$ the function $F(\cdot+iy)=\F^{-1}(\widehat{f}e^{-(y|\cdot)})$ belongs to $L^p(V).$ Moreover,
 \Bea\label{4.15}
 \|F(\cdot+iy)\|_p\lesssim Q^{-\frac \mu s}(y)\|f\|_{B_\mu^{p,s}}
 \Eea
 with constants independent of $f$ or $y\in\O.$
 \end{lem}

In the proofs, we denote $\{\chi_j\}$ a family of functions defined as $\widehat \chi_j (\xi):=\widehat \chi (g_j^{-1} (\xi)$ from an arbitrary $\widehat \chi\in C_c^\infty (B_4 (\mathbf e))$ so that $0\leq \widehat \chi \leq 1$ and $\widehat \chi$ is identically 1 in $B_2 (\mathbf e).$ We shall use the following estimate (formula (3.47) of \cite{BBGR}): there exist two positive numbers $C$ and $\gamma$ such that
\begin{equation}\label{3.47}
||\mathcal F^{-1} (\widehat \chi_j e^{-(y\vert \cdot)}||_1 \leq Ce^{-\frac {(g_j y\vert \mathbf e)}\gamma}
\end{equation}

 \begin{proof}[Proof of Lemma 6.4]
\hskip 2truemm By homogeneity (see \cite[Proposition 3.19 ]{DD}), it is sufficient to prove (\ref{4.15}) when $y=\eta{\bf e},$ for some fixed $\eta>0$ to be chosen below.
 Let us denote $\widehat{g}=\widehat{f}e^{-\eta({\bf e}|\cdot)},$ so that $g=\sum\limits_jg*\psi_j\in\mathcal{S}_{\O}.$ Applying (\ref{4.12}) to the sequence $\{f_j = g*\psi_j\}$
 and using the Young inequality, we obtain
 \Beas
 \|F(\cdot+i\eta{\mathbf e})\|_p &=&\|g\|_p=\|\sum_jg*\psi_j\|_p\\
 &\lesssim& \left[\sum_j(Q^*)^{-\mu}(\xi_j)e^{\da(\xi_j|{\bf e})}\|(g*\psi_j\|^s_p\right]^\frac{1}{s} \\
 &\lesssim& \left[\sum_j(Q^*)^{-\mu}(\xi_j)e^{\da(\xi_j|{\bf e})}\|\F^{-1}(\widehat{f}\hskip 1truemm \widehat{\psi_j}e^{-\eta({\bf e}|\cdot)})\|^s_p\right]^\frac{1}{s}\\
 &\lesssim& \left[\sum_j(Q^*)^{-\mu}(\xi_j)e^{\da(\xi_j|{\bf e})}\|f*\psi_j\|_p^s\|\F^{-1}(e^{-\eta({\bf e}|\cdot)}\widehat{\chi_j})\|^s_1\right]^\frac{1}{s}.\\
 \Eeas
 Now, $\|\F^{-1}(e^{-\eta({\mathbf e}|\cdot)}\widehat{\chi_j})\|_1$ is bounded by a constant times $e^{-\ga\eta(\xi_j|{\bf e})}$ by formula (\ref
{3.47}). Therefore,
 \Beas
\|F(\cdot+i\eta{\mathbf e})\|_p&\lesssim& \left[\sum_j(Q^*)^{-\mu}(\xi_j)e^{\da(\xi_j|{\bf e})-\ga\eta(\xi_j|{\bf e})}\|f*\psi_j\|_p^s\right]^\frac{1}{s}
 \Eeas
 We only need to choose $\eta$ larger than $\da/\ga.$
\end{proof}

We now conclude the proof of Theorem \ref{4.11}. Given $f\in\mathcal{D}_{\O}$ and $F:=\mathcal C f,$ Lemma \ref{4.14} applied to
$\F^{-1}(\widehat{f}e^{-(y|{\bf e}}))$ gives us
\Beas
\|F(\cdot+i2y)\|_p &\lesssim& Q^{-\frac \mu s}(y)\left[\sum_j(Q^*)^{-\mu}(\xi_j)\|\F^{-1}(\widehat{f}\hskip 1truemm \widehat{\psi_j}e^{-(y|\cdot)})\|^s_p\right]^\frac{1}{s}\\
 &\lesssim&Q^{-\frac \mu s}(y)\left[\sum_j(Q^*)^{-\mu}(\xi_j)e^{-\ga(\xi_j|y)}\|f*\psi_j\|_p^s\right]^\frac{1}{s},
\Eeas
where we have used formula (\ref{3.47}) and Young's inequality again. Thus
\Beas
I&:=&\int_\O\|F(\cdot+i2y)\|_p^qQ^{\nu-\frac nr}(y)dy\\
&\lesssim&\int_\O Q^{-\frac{\mu q}{s}}(y)\left[\sum_j(Q^*)^{-\mu}(\xi_j)e^{-\ga(\xi_j|y)}\|f*\psi_j\|_p^s\right]^\frac{q}{s}Q^{\nu-\frac nr}(y)dy.
\Eeas

When $q/s\leq 1$ i.e. $q\leq s,$ then
\Beas
I&\lesssim&\int_\O Q^{-\frac{\mu q}{s}}(y)\sum_j(Q^*)^{-\mu\frac{q}{s}}(\xi_j)e^{-\ga\frac{q}{s}(\xi_j|y)}\|f*\psi_j\|_p^qQ^{\nu-\frac nr}(y)dy\\
&\lesssim&\sum_j(Q^*)^{-\mu\frac{q}{s}}(\xi_j)\|f*\psi_j\|_p^q\int_\O e^{-\ga\frac{q}{s}(\xi_j|y)}Q^{-\frac{\mu q}{s}+\nu-\frac nr}(y)dy\\
&\lesssim&\sum_j(Q^*)^{-\mu\frac{q}{s}+\frac{\mu q}{s}-\nu}(\xi_j)\|f*\psi_j\|_p^q=\sum_j(Q^*)^{-\nu}(\xi_j)\|f*\psi_j\|_p^q
\Eeas
provided $-\mu_j\frac{q}{s}+\nu_j>\frac{m_j}{2},\,\,j=1,\ldots,r$ thanks to Lemma \ref{integ}. This leads to  condition (i).

Assume now that $q/s> 1$ i.e. $q>s.$ Let $\beta=(\ba_1,\ldots,\ba_r)\in\R^r$ be a vector to be chosen below. We have
\Beas
I&\lesssim&\int_\O Q^{-\frac{\mu q}{s}}(y)\left[\sum_j(Q^*)^{-\mu-\ba}(\xi_j)\|f*\psi_j\|_p^s(Q^*)^{\ba}(\xi_j)e^{-\ga(\xi_j|y)}\right]^\frac{q}{s}Q^{\nu-\frac nr}(y)dy
\Eeas
Then by H\"older's inequality,
\Bea
\left[\sum_j(Q^*)^{-\mu-\ba}(\xi_j)\|f*\psi_j\|_p^s(Q^*)^{\ba}(\xi_j)e^{-\ga(\xi_j|y)}\right]^\frac{q}{s}\leq\\{}{}\nonumber
\left[\sum_j(Q^*)^{-\frac{q}{s}(\mu+\ba)}(\xi_j)\|f*\psi_j\|_p^qe^{-\ga(\xi_j|y)}\right]
\left[\sum_j(Q^*)^{\ba\left(\frac{q}{s}\right)'}(\xi_j)e^{-\ga(\xi_j|y)}\right]^{\left(\frac{q}{s}\right)/\left(\frac{q}{s}\right)'}.
\Eea
But
$$\sum_j(Q^*)^{\ba(\frac{q}{s})'}(\xi_j)e^{-\ga(\xi_j|y)}\sim\int_{\O^*}e^{-\ga(\xi|y)}(Q^*)^{\ba\left(\frac{q}{s}\right)'-\frac nr}(\xi)d\xi
\sim Q^{-\ba\left(\frac{q}{s}\right)'}(y)$$
provided $\ba_j \left(\frac{q}{s}\right)'>\frac{n_j}{2},\,\,j=1,\ldots,r$ thanks to Lemma \ref{integ}. It follows that
\Beas
I&\lesssim& \int_\O Q^{-\frac{\mu q}{s}-\ba\left(\frac{q}{s}\right)'}(y)\left[\sum_j(Q^*)^{-\frac{q}{s}(\mu+\ba)}(\xi_j)\|f*\psi_j\|_p^qe^{-\ga(\xi_j|y)}\right]Q^{\nu-\frac nr}(y)dy \\
&\lesssim&\sum_j(Q^*)^{-\frac{q}{s}(\mu+\ba)}(\xi_j)\|f*\psi_j\|_p^q \int_\O e^{-\ga(\xi_j|y)}Q^{-\frac{q}{s}(\mu+\ba)+\nu-\frac nr}(y)dy\\
\Eeas
 The last integral converges if $-\frac{q}{s}(\mu_j+\ba_j)+\nu_j>\frac{m_j}{2},\,\,j=1,\ldots,r$ thanks to Lemma \ref{integ}. It follows that
 \Beas
  I&\lesssim&\sum_j(Q^*)^{-\nu}(\xi_j)\|f*\psi_j\|_p^q
 \Eeas
 whenever $\ba_j \left(\frac{q}{s}\right)'>\frac{n_j}{2},\,\,j=1,\ldots,r$ and $-\frac{q}{s}(\mu_j+\ba_j)+\nu_j>\frac{m_j}{2},\,\,j=1,\ldots,r.$ That is
 we will conclude with $I:=\int_\O\|F(\cdot+i2y)\|_p^qQ^{\nu-\tau}(y)dy\lesssim\|f\|_{B_\nu^{p,q}}$ if we can choose real numbers $\ba_j$ so that
 $$\frac{n_j}{2}\left(1-\frac{1}{q/s}\right)<\ba_j<\frac{1}{q/s}\left(\nu_j-\frac{m_j}{2}\right)-\mu_j,\,\,j=1,\ldots,r.$$
 Solving for $q/s$ we get
 $$\frac qs<\min_{j=1,\ldots,r}\frac{\nu_j-\frac{m_j}{2}+\frac{n_j}{2}}{\mu_j+\frac{n_j}{2}}$$
 which gives condition (ii).




From Theorems 6.1 and 6.2, we deduce the following corollary for tube domains over symmetric cones.

\begin{cor}
Let $\nu=(\nu_1,\cdots,\nu_r)\in \mathbb R^r$ such that $\nu_j >\frac {m_j}2, \hskip 2truemm j=1,\cdots,r,$  $1<p<p_\nu$ and $1<s<\infty.$ Assume further that there is a positive number $\delta$ and a vector $\mu = (\mu_1,\cdots,\mu_r)\in \mathbb R^r$ with $\mu_j >0, \hskip 2truemm j=1,\cdots,r$ and a constant $C=C(\mu, \delta)>0$ such that 
$$\left\|\sum_j f_j\right\|_p\leq C\left[\sum_j(Q^*)^{-\mu}(\xi_j)e^{\da(\xi_j|{\bf e})}\|f_j\|^s_p\right]^\frac{1}{s}$$
holds for every finite sequence $\{f_j\}\in L^p (V)$ with $\mbox Supp \hskip 2truemm \widehat f_j \in B_2 (\xi_j).$ We assume that for the index $q,$ we are in one of the following two situations.
\begin{enumerate}
\item[(a)]
If $q'_\nu (p) < s,$ then there are two cases:
\begin{enumerate}
\item[(i)]
$q'_\nu (p) < q \leq s$ and $q < \min \left \{s\min \limits_{\hskip 2truemm j=1,\cdots,r} \frac {\nu_j -\frac {m_j}2}{\mu_j},  \widetilde q_{\nu, p}\right \};$
\item[(ii)]
$s < q < \min \left \{s\min \limits_{\hskip 2truemm j=1,\cdots,r} \frac {\nu_j - \frac {m_j}2 +\frac {n_j}2}{\mu_j +\frac {n_j}2},  \widetilde q_{\nu, p}\right \}.$
\end{enumerate}
\item[(b)]
If $q'_\nu (p) \geq s,$ then $$q'_\nu (p) <q< \min \left \{s\min \limits_{\hskip 2truemm j=1,\cdots,r} \frac {\nu_j - \frac {m_j}2 +\frac {n_j}2}{\mu_j +\frac {n_j}2}, \widetilde q_{\nu, p}\right \}.$$
\end{enumerate}
Then the Bergman projector $P_\nu$ admits a bounded extension from $L^{p, q}_\nu (T_\Omega)$ to $A^{p, q}_\nu (T_\Omega).$
\end{cor}

\subsection{The proof of Theorem 2.3}
We refer to \cite[section 5]{BBGR}. Let $\Omega = \Lambda_n$ denote the Lorentz cone in $\mathbb R^n, \hskip 2truemm n\geq 3,$ defined by
$$\Lambda_n = \{y=(y_1,y') \in \mathbb R^n: \hskip 1truemm \Delta_1 (y) >0, \hskip 2truemm \Delta_2 (y) >0\},$$
with $\Delta_1 (y)=y_1$ and $\Delta_2 (y) =y_1^2 -|y'|^2.$ The rank of $\Lambda_n$ is $r=2.$ For $j\geq 1,$ take a maximal $2^{-j}$-separated sequence $\{\omega_k^{(j)}\}_{k=1}^{k_j}$ of points of the sphere $\mathbb S^{n-2} \subset \mathbb R^{n-1},$ with respect to the Euclidean distance (so that $k_j \sim 2^{j(n-2)}).$ Then define the sets\\
$$E_{j, k} =$$
$$ \left \{(\xi_1, \xi')\in \Lambda_n: 2^{-1} <\tau <2, \hskip 2truemm 2^{-2j-2} <1-\frac {|\xi'|^2}{\xi_1^2} <2^{-2j+2}
{\rm and} \hskip 1truemm \left |\frac {\xi'}{|\xi'|} - \omega_k^{(j)}\right | \leq \delta 2^{-j} \right \}$$
where the constant $\delta >0$ is suitably chosen \cite{BBGR}.

Recall that for $\Lambda_n$ we have $r=2,\,\,m_j=(2-j)d,\,\,n_j=(j-1)d$ with $d=n-2.$ Thus for $\nu=(\nu_1,\nu_2)\in\mathbb{R}^2$
such that $\nu_1>\frac n2-1,\,\,\nu_2>0,$ we have
$$p_\nu =1+\frac {\nu_2 +\frac n2}{\left(\frac n2 -1 -\nu_2\right)_+};$$
$$q_\nu= 1+\min \limits_{1\leq i \leq 2} \frac {\nu_i - \frac {m_i}2}{\frac {n_i}2}=1+\frac {\nu_2}{\frac {n}2-1};$$
$$q'_\nu (p)= \frac {\nu_2 +\frac n2 -1}{\nu_2 +\left(1-\frac 1{p_\sharp}\right)\left(\frac n2 -1\right)};$$
$$\tilde q_{\nu, p} =\min \limits_{1\leq i \leq 2} \frac {\nu_j + (j-1)\frac d2}{\left(\frac n{2p'} -1-(2-j)\frac d2\right)_+}
=\frac {\nu_2 + \frac n2-1}{\left(\frac n{2p'} -1\right)_+}.$$
The following theorem is a consequence of \cite[Proposition 5.5]{BBGR} and Corollary 6.5 above.

\begin{thm}\label{main}
Let $1\leq p < p_\nu, \hskip 2truemm 1\leq s <\infty.$ Suppose that for some $\mu \geq 0$ 
there exists a constant $C_{\mu}$ such that
\begin{equation}\label{dec}
\left \Vert\sum_{k=1}^{k_j} f_k \right \Vert \leq  C_{\mu} 2^{\frac {2j{\mu}}{s}}\left [\sum_{k=1}^{k_j} \Vert f_k \Vert_p^s\ \right ]^{\frac 1s}
\quad {\rm for \hskip 2truemm all} \hskip 2truemm j\geq 1,
\end{equation}
for every sequence $\{f_k\}$ satisfying $Supp \hskip 1truemm \widehat f_k \subset E_{j, k}.$
We assume that for the index $q,$ we are in one of the following two situations.
\begin{enumerate}
\item[(a)]
If $q'_\nu (p) < s,$ then there are two cases:
\begin{enumerate}
\item[(i)]
$q'_\nu (p) < q \leq s$ and $q < \min \left \{s\min \limits_{\hskip 2truemm j=1,\cdots,r} \frac {\nu_j -\frac {m_j}2}{\mu},  \widetilde q_{\nu, p}\right \};$
\item[(ii)]
$s < q < \min \left \{s\min \limits_{\hskip 2truemm j=1,\cdots,r} \frac {\nu_j - \frac {m_j}2 +\frac {n_j}2}{\mu+\frac {n_j}2},  \widetilde q_{\nu, p}\right \}.$
\end{enumerate}
\item[(b)]
If $q'_\nu (p) \geq s,$ then $$q'_\nu (p)  < q< \min \left \{s\min \limits_{\hskip 2truemm j=1,\cdots,r} \frac {\nu_j - \frac {m_j}2 +\frac {n_j}2}{\mu +\frac {n_j}2}, \widetilde q_{\nu, p}\right \}.$$
\end{enumerate}
Then $P_\nu$ is bounded in $L^{p, q}_\nu.$
\end{thm}

The following theorem is a consequence of the $l^2$-decoupling theorem recently proved by Bourgain and Demeter \cite{BD}.

\begin{thm}\label{bourgain}
The estimate {\rm (\ref{dec})} is valid for $s=2$ and the following values of $p$ and $\mu:$
\begin{enumerate}
\item
$2\leq p \leq \frac {2n}{n-2}$ and $\mu=0;$
\item
$p\geq \frac {2n}{n-2}$ and $\mu =\frac {n-2}2 - \frac np.$
\end{enumerate}
\end{thm}

We deduce the following corollary.

\begin{cor}\label{resultat}
Let $n\geq 3$ and $\nu=(\nu_1,\nu_2)\in \mathbb{R}^2$ such that $\nu_1>\frac n2-1,\,\,\nu_2>0.$
The weighted Bergman projector $P_\nu$ is bounded in $L^{p, q}_\nu(T_{\Lambda_n})$ for the following values of $p, q$ and $\nu:$
\begin{enumerate}
\item
$2\leq p \leq \frac {2n}{n-2}, \hskip 2truemm p<p_\nu$ and $q'_\nu (p) < q < 2q_\nu;$
\item
$p_\nu >p  \geq \frac {2n}{n-2}, \hskip 2truemm q'_\nu (p) < q<\min \left \{2\frac {\nu_1 - \frac n2 +1}{\frac n2 -1 -\frac np},  \widetilde q_{\nu, p}\right \}$ provided $0 < \nu_2 <  \frac n2 -1;$
\item
$p_\nu >p  > \frac {2n}{n-2}$ and $2 < q< \min \left \{2\frac {\nu_1 - \frac n2 +1}{\frac n2 -1 -\frac np},    \widetilde q_{\nu, p}\right \}$ provided  $\nu_2 \geq \frac n2 -1.$
\end{enumerate}
\end{cor}

\begin{proof}
We first notice that if $p\geq 2,$ the following equivalence holds
\begin{equation}
q'_\nu (p) <2 \quad {\rm if \hskip 2truemm and \hskip 2truemm only \hskip 2truemm if}\quad \nu_2 > \left(\frac n2 -1\right)\left(1-\frac 2p\right).
\end{equation}
Also,
\begin{equation}
\widetilde q_{\nu, p} \geq 2q_\nu >2 \quad {\rm for \hskip 2truemm all}\quad 1\leq p \leq \frac {2n}{n-2} \quad {\rm and} \quad \nu_2 >0.
\end{equation}
\vskip 2truemm
1) Suppose first that $2\leq p \leq \frac {2n}{n-2}$ and $\mu=0.$ Then by Theorem 6.7, estimate (40) is satisfied for $s=2.$ We distinguish two cases.\\
{\underline {Case}} 1. We suppose that $\nu_2 > \left(\frac n2 -1\right)\left(1-\frac 2p\right).$ Then by equation (41), we have $q'_\nu (p) <s=2.$ It follows from Theorem 6.6 that $P_\nu$ is bounded in $L^{p, q}_\nu (T_{\Lambda_n})$ if\\
$\bullet
\quad q'_\nu (p) <q\leq 2$ and $q<\widetilde q_{\nu, p}$\\
$\bullet
\quad 2<q <\min \{2q_\nu, \widetilde q_{\nu, p}\}=2q_\nu$\\
and the last equality above follows from equation (42).\\
{\underline {Case}} 2. We suppose that $0<\nu_2 \leq \left(\frac n2 -1\right)\left(1-\frac 2p\right).$ Then by (1), we have $q'_\nu (p) \geq s=2.$ It follows from Theorem 6.6 that $P_\nu$ is bounded in $L^{p, q}_\nu (T_{\Lambda_n})$ if
$q'_\nu (p) <q<\min \{2q_\nu, \widetilde q_{\nu, p}\}=2q_\nu$
and the last equality above again follows from equation (42). This proves the assertion (1) of the corollary.
\vskip 2truemm
2) Observe that $p=\frac {2n}{n-2} $ means that $\mu=0.$ We assume then that $p>\frac {2n}{n-2}.$ Notice that
$$p_\nu = 1+\frac {\nu_2 +\frac n2}{(\frac n2 -1-\nu_2)_+}=\left \{
\begin{array}{clcr}
1+\frac {\nu_2 +\frac n2}{\frac n2 -1-\nu_2}&\rm if &0<\nu_2 <\frac n2 -1\\
\infty & &\rm elsewhere
\end{array}
\right
.
.$$
This suggests two cases: $0<\nu_2 <\frac n2 -1$ and $\nu_2 \geq \frac n2 -1.$ Also, the case $p_\nu \leq \frac {2n}{n-2}$ is irrelevant since we must have $p<p_\nu:$ it would refer to assertion (1) of the corollary. This suggests to replace the first case by $\frac {n-2}{2n} <\nu_2 <\frac n2 -1.$ Furthermore, note that $p>\frac {2n}{n-2}$ is equivalent to $\mu=\frac n2 -1-\frac np >0.$ Also, $p>\frac {2n}{n-2}$ implies that $\frac n{2p'} -1 >\frac {n-2}{4}>0$ and so
$$\widetilde q_{\nu, p} = 2\frac {\nu_2 +\frac n2 -1}{n-2-\frac np}.$$
Moreover, we check that $\widetilde q_{\nu, p} >2$ in the following two cases:
\begin{itemize}
\item
$\frac{n-2}{2n}<\nu_2 <\frac n2 -1$ and $p_\nu >p>\frac {2n}{n-2};$
\item
$\nu_2 \geq \frac n2 -1$ and $p>\frac {2n}{n-2}.$
\end{itemize}
{\underline {Case}} 1. We suppose first that $\frac {n-2}{2n} <\nu_2 <\frac n2 -1, \quad p_\nu >p>\frac {2n}{n-2}$ and $\mu=\frac n2 -1-\frac np.$ 
By Theorem 6.7, estimate (40) is satisfied for $s=2.$ We check easily that $\frac {n-2}{2n} < (\frac n2 -1)(1-\frac 2p)$ whenever $p>\frac {2n}{n-2}.$ According to equation (1), we distinguish two subcases:
\begin{enumerate}
\item[(i)]
If $\frac {n-2}{2n} <\nu_2 \leq \left(\frac n2 -1\right)\left(1-\frac 2p\right),$ then $\quad q'_\nu (p) \geq s=2.$ It follows from Theorem 6.6 that $P_\nu$ is bounded in $L^{p, q}_\nu (T_{\Lambda_n})$ if
$$\quad q'_\nu (p) < q< \min \left \{2\min \left \{\frac {\nu_1 -\frac n2 +1}{\frac n2 -1-\frac np}, \frac {\nu_2 +\frac n2 -1}{n-2-\frac np}\right \}, \widetilde q_{\nu, p}\right \}=\min \left \{2\frac {\nu_1 - \frac n2 +1}{\frac n2 -1 -\frac np},  \widetilde q_{\nu, p}\right \}.$$
\item[(ii)]
If $\left(\frac n2 -1\right)\left(1-\frac 2p\right) <\nu_2 < \frac n2 -1,$ then $\quad q'_\nu (p) < s=2.$ It follows from Theorem 6.6 that $P_\nu$ is bounded in $L^{p, q}_\nu (T_{\Lambda_n})$ if\\
$\bullet
\quad q'_\nu (p) <q\leq 2$ and $q<\min \left \{2\min \left \{\frac {\nu_1 -\frac n2 +1}{\frac n2 -1-\frac np}, \frac {\nu_2 }{\frac n2-1-\frac np}\right \}, \widetilde q_{\nu, p}\right \}=\min \left \{2\frac {\nu_1 - \frac n2 +1}{\frac n2 -1 -\frac np},  \widetilde q_{\nu, p}\right \}$\\
$\bullet
\quad 2<q <\min \left \{2\min \left \{\frac {\nu_1 -\frac n2 +1}{\frac n2 -1-\frac np}, \frac {\nu_2 +\frac n2 -1}{n-2-\frac np}\right \}, \widetilde q_{\nu, p}\right \}=\min \left \{2\frac {\nu_1 - \frac n2 +1}{\frac n2 -1 -\frac np},  \widetilde q_{\nu, p}\right \}.$\\
\end{enumerate}
This proves the assertion (2) of the corollary.\\
{\underline {Case}} 2. We suppose that $\nu_2 \geq \frac n2 -1, \quad p_\nu =\infty >p>\frac {2n}{n-2}$ and $\mu=\frac n2 -1-\frac np.$
By Theorem 6.7, estimate (40) is satisfied for $s=2.$ In this case, $q'_\nu (p) <s=2$ since $\nu_2 \geq \frac n2 -1.$ It follows from Theorem 6.6 that $P_\nu$ is bounded in $L^{p, q}_\nu (T_{\Lambda_n})$ if\\
$\bullet
\quad q'_\nu (p) <q\leq 2$ and $q<\min \left \{2\min \left \{\frac {\nu_1 -\frac n2 +1}{\frac n2 -1-\frac np}, \frac {\nu_2 }{\frac n2-1-\frac np}\right \}, \widetilde q_{\nu, p}\right \}=\min \left \{2\frac {\nu_1 - \frac n2 +1}{\frac n2 -1 -\frac np},  \widetilde q_{\nu, p}\right \}$\\
$\bullet
\quad 2<q <\min \left \{2\min \left \{\frac {\nu_1 -\frac n2 +1}{\frac n2 -1-\frac np}, \frac {\nu_2 +\frac n2 -1}{n-2-\frac np}\right \}, \widetilde q_{\nu, p}\right \}=\min \left \{2\frac {\nu_1 - \frac n2 +1}{\frac n2 -1 -\frac np},  \widetilde q_{\nu, p}\right \}.$
This proves the assertion (3) of the corollary.
\end{proof}

Assertion (1) of the previous corollary is just assertion (2) of Theorem 2.3. Assertions (1) and (2) of Theorem 2.3 are particular cases of  \cite[Corollary 1.4]{DD1} for tube domains over Lorentz cones with $\mu=\nu.$ For assertions (3) and (4),   for $p_\nu >p>\frac {2n}{n-2},$ we obtain by interpolation the following result which is sharper than assertions (2) and (3) of the previous corollary.

\begin{cor}
Let $n\geq 3$ and $\nu=(\nu_1, \nu_2)\in \mathbb R^2$ such that $\nu_1 >\frac n2 -1, \hskip 2truemm \nu_2 >0.$ The weighted Bergman projector $P_\nu$ is bounded in $L^{p, q}_\nu (T_{\Lambda_n})$ for the following values of $p, q$ and $\nu.$
\begin{enumerate}
\item
$\frac {n-1}{\frac n2 -1 - \nu_2} > p > \frac {2n}{n-2}$ and $\quad 2<q< \widetilde q_{\nu, p}$ provided 
$\frac {n-2}{2n}<\nu_2 <\frac n2 -1;$
\item
$ p > \frac {2n}{n-2}$ and $\quad 2<q< \widetilde q_{\nu, p}$ provided $\nu_2 \geq \frac n2 -1.$
\end{enumerate}
\end{cor}

\begin{proof}
\hskip 2truemm The situation is represented in Figure \ref{fig6} (Figure 1.1 of [2]) . From assertion (1) of Theorem 2.3, 
we obtain that $P_\nu$ is bounded in $L^{p, q}_\nu (T_{\Lambda_n})$ if
$$0\leq \frac 1p \leq 1 \quad {\rm and} \quad \frac 1{q_\nu}<\frac 1q<\frac 1{q'_\nu}.$$
Combining with assertion (1) of Corollary 6.8, we deduce by interpolation that $P_\nu$ is bounded in $L^{p, q}_\nu 
(T_{\Lambda_n})$ for
$$\frac 1{p_\nu} < \frac 1p < \frac {n-2}{2n} \quad {\rm and} \quad \frac {n-2}{2n}<\nu_2 <\frac n2 -1$$
$$({\rm resp.} \quad \frac 1p < \frac {n-2}{2n} \quad {\rm and} \quad \nu_2 \geq \frac n2 -1),$$
if the couple $(\frac 1p, \frac 1q)$ lies in the triangle given by the inequalities
\begin{equation}\label{ineq}
y>\frac {n-2}n (-q_\nu x +1), \quad \frac 1{2q_\nu} <x<\frac 1{q_\nu}.
\end{equation}
Remind that for such values of $p,$ we have $\frac n{2p'} -1>0$ and so $\widetilde q_{\nu, p} = 
2\frac {\nu_2 +\frac n2 -1}{n-2-\frac np}.$ It is now easy to conclude that the first inequality in (\ref{ineq}) 
can be written in the form
 $$q < \widetilde q_{\nu, p}.$$
\end{proof}

\begin{remark}
{\rm According to Theorem 2.3, the conjecture stated in the introduction of [2] for $\nu_1=\nu_2$ 
is valid for tube domains over Lorentz cones. More precisely, the weighted Bergman projector $P_\nu$ is
bounded in $L^{p, q}_\nu$  when the couple $(\frac 1p, \frac 1q)$ lies in the blank region of Figure 1.1 of [2] 
depicted below. This result has been proved in the blue region in [4] and in the red region in [2] for $\nu_1=\nu_2$. 
The result in the blank region is given by Theorem 2.3.  In particular, the case $\nu_1=\nu_2=\frac n2$  and $p=q$ in 
Theorem 2.3 corresponds to} \cite[Theorem 1.2]{BoNa}.
\end{remark}

\vskip 60truemm

\begin{figure}[tbph]
\centering
\includegraphics[width=0.7\linewidth]{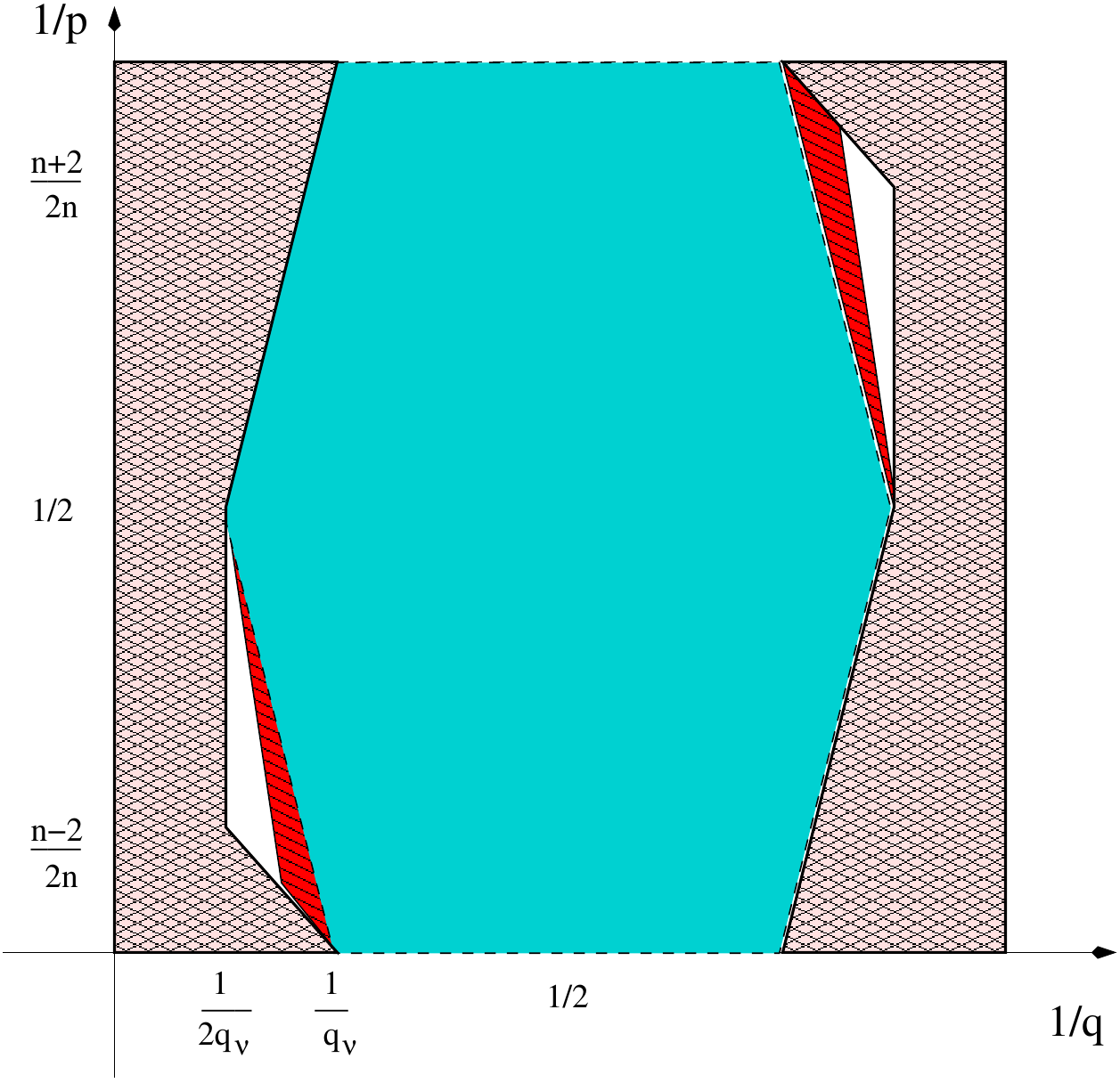}
\caption{Region of boundedness of $P_\nu$ for Lorentz cones and $\nu_1=\nu_2$}
\label{fig6}
\end{figure}

Finally, the proof of Theorem 2.4 is just a combination of the Theorem 2.3 for $n=3$ and Theorem 2.1 for the Pyateckii-Shapiro domain.
\vskip 2truemm
\noindent
\textbf{Acknowledgements.} The authors wish to express their gratitude to Aline Bonami and Gustavo Garrig\'os for valuable discussions.


\end{document}